\documentclass[11pt, reqno]{amsart}
\usepackage{amsmath,amsfonts,amssymb,amsthm,color,cancel}
\usepackage{epsfig}
\usepackage[normalem]{ulem}
\usepackage{ulem}

\usepackage{bbm}
\usepackage{enumerate}
\usepackage{amstext}
\usepackage{mathrsfs}
\usepackage{xspace}
\usepackage{amsfonts}
\usepackage{amsmath}
\usepackage{amssymb}
\usepackage{amstext}
\usepackage{amsthm}       
\usepackage{xspace}
\usepackage[active]{srcltx}
\usepackage{cite}

\newtheorem{theorem}{Theorem}
\newtheorem{lemma}{Lemma}[section]
\newtheorem{definition}[lemma]{Definition}
\newtheorem{proposition}[lemma]{Proposition}
\newtheorem{remark}[lemma]{Remark}
\newtheorem{corollary}[lemma]{Corollary}

\newcommand{\be}{\begin{equation}}
\newcommand{\ee}{\end{equation}}
\newcommand{\gm}{\gamma}
\newcommand{\R}{\mathbb{R}}

\newcommand{\eps}{\varepsilon}

\newcommand{\bp}{\begin{proof}}
\newcommand{\ep}{\end{proof}}

\numberwithin{equation}{section}

\begin{document}




\title{A representation formula of the viscosity solution of the contact Hamilton-Jacobi equation and its applications}

\author{Panrui Ni \and Lin Wang \and Jun Yan}
\address{Shanghai Center for Mathematical Sciences,  Fudan University, Shanghai 200438, China}

\email{prni18@fudan.edu.cn}
\address{School of Mathematics and Statistics, Beijing Institute of Technology, Beijing 100081, China}

\email{lwang@bit.edu.cn}
\address{School of Mathematical Sciences, Fudan University, Shanghai 200433, China}

\email{yanjun@fudan.edu.cn}

\subjclass[2010]{37J50; 35F21; 35D40.}
\keywords{Weak KAM theory, Hamilton-Jacobi equations, Aubry sets}

\begin{abstract}
\noindent Assume $M$ is a closed, connected and smooth Riemannian manifold. We consider the evolutionary Hamilton-Jacobi equation
\begin{equation*}
  \left\{
   \begin{aligned}
   &\partial_t u(x,t)+H(x,u(x,t),\partial_xu(x,t))=0,\quad (x,t)\in M\times(0,+\infty),\\
   &u(x,0)=\varphi(x),\\
   \end{aligned}
   \right.
\end{equation*}
where $\varphi\in C(M)$ and the stationary one
\begin{equation*}
  H(x,u(x),\partial_x u(x))=0,
\end{equation*}
where $H(x,u,p)$ is continuous, convex and coercive in $p$, uniformly Lipschitz in $u$. By introducing a solution semigroup, we provide a {\it representation formula} of the viscosity solution of the evolutionary equation. As its applications, we obtain a necessary and sufficient condition for the existence of the viscosity solutions of the stationary equations.  Moreover, we prove a  new comparison theorem depending on the neighborhood of the projected Aubry set essentially, which is   different from the one for the Hamilton-Jacobi equation independent of $u$.
\end{abstract}

\date{\today}
\maketitle

\tableofcontents

\section{Introduction and main results}

The study of the theory of viscosity solutions of the following two forms of Hamilton-Jacobi equations
\begin{equation}\label{evll}
  \partial_t u(x,t)+H(x,u(x,t),\partial_xu(x,t))=0,
\end{equation}
and
\begin{equation}\label{hj}
  H(x,u(x),\partial_x u(x))=0
\end{equation}
has a long history. There are many celebrated results on the existence, uniqueness, stability and large time behavior problems for the viscosity solutions (see \cite{Intr,CEL,CHL2,Vis} for instance).

For the case with the Hamiltonian independent of  $u$, its characteristic equation is the Hamilton equation.
For the case with the Hamiltonian depending on $u$, the  characteristic equation is called the contact Hamilton equation. In \cite{Wa1}, the authors introduced an implicit variational principle for the contact Hamilton equation. Based on that, a representation formula was provided for the  viscosity solution of the evolutionary equation, and the existence of the solutions for the ergodic problem was also proved in \cite{Wa2}. In \cite{Wa3}, the Aubry-Mather theory was developed  for contact Hamiltonian systems with strictly increasing dependence on $u$. In \cite{Wa4}, the authors further studied the strictly decreasing case, and discussed  large time behavior of the viscosity solution of the evolutionary case.

 All of the results in \cite{Wa2,Wa3,Wa4} are  based on the implicit variational principle established in \cite{Wa1}.  In order to get the $C^1$-regularity of the minimizers, it was assumed that  the contact Hamiltonian $H$ is  {\it of class $C^3$, strictly convex} and {\it superlinear}.  This paper is devoted to reducing those assumptions to {\it continuous, convex} and {\it coercive}, which are standard from the PDE aspect. It is clear that the contact Hamiltonian equation can not be defined under these assumptions. Moreover, the compactness estimate of the set of certain minimizers does not hold true. That estimate (see \cite[Lemma 2.1]{Wa2}) plays a crucial role in the previous work \cite{Wa2,Wa3,Wa4}. For the classical Hamilton-Jacobi (HJ) equation with time-independence, the related problems were considered in \cite{gen,Fa}.
Different from them, one has to face certain new difficulties  due to  the appearance of the Lavrentiev phenomenon caused by time-dependence.

By combining dynamical and PDE approaches,  we provide a {\it representation formula} of the viscosity solution of the evolutionary equation, which can be referred to as an implicit Lax-Oleinik semigroup.
As its applications, we obtain a necessary and sufficient condition for the existence of the viscosity solutions of the stationary equations.  It is well known that the comparison theorem plays a central role in the viscosity solution theory. We prove a new comparison result depending on a  neighborhood of the projected Aubry set essentially. An example is constructed to show that the requirement of the neighborhood  is {\it necessary}   for a special class of Hamilton-Jacobi equations that do not satisfy  the ``proper" condition introduced in  \cite{CHL2}. Comparably, the viscosity solution is determined completely by the projected Aubry set itself for the ``proper" cases (\cite[Theorem 1.6]{WWY4}).

Throughout this paper,  we assume $M$ is a closed (compact without boundary), connected and smooth Riemannian manifold and $H:T^*M\times\mathbb R\rightarrow\mathbb R$ satisfies
\begin{itemize}
\item [\textbf{(C):}] $H(x,u,p)$ is continuous;

\item [\textbf{(CON):}] $H(x,u,p)$ is convex in $p$, for any $(x,u)\in M\times \mathbb R$;

\item [\textbf{(CER):}] $H(x,u,p)$ is coercive in $p$, i.e. $\lim_{\|p\|\rightarrow +\infty}(\inf_{x\in M}H(x,0,p))=+\infty$;

\item [\textbf{(LIP):}] $H(x,u,p)$ is Lipschitz in $u$, uniformly with respect to $(x,p)$, i.e., there exists $\lambda>0$ such that $|H(x,u,p)-H(x,v,p)|\leq \lambda|u-v|$, for all $(x,p)\in\ T^*M$ and all $u,v\in\mathbb R$.
\end{itemize}
Correspondingly, one has the  Lagrangian associated to $H$:
\begin{equation*}
  L(x,u,\dot x):=\sup_{p\in T^*_xM}\{\langle \dot x,p\rangle-H(x,u,p)\}.
\end{equation*}
A list of notations is provided at the end of this section.
\begin{remark}\label{coerh}
Due to the absence of superlinearity of $H$, the  Lagrangian $L$ may take the value $+\infty$. Define
\[\text{dom}(L):=\{(x,\dot{x},u)\in TM\times\R\ |\ L(x,u,\dot{x})<+\infty\}.\]
Then $L$ satisfies the following properties (see \cite[Proposition 2.7]{gen} for instance)
\begin{itemize}
\item [\textbf{(LSC):}] $L(x,u,\dot x)$ is lower semicontinuous, and continuous on the interior of dom($L$);

\item [\textbf{(CON):}] $L(x,u,\dot x)$ is convex in $\dot x$, for any $(x,u)\in M\times \mathbb R$;

\item [\textbf{(LIP):}] $L(x,u,\dot x)$ is Lipschitz in $u$, uniformly with respect to $(x,\dot x)$, i.e., there exists $\lambda>0$ such that $|L(x,u,\dot x)-L(x,v,\dot x)|\leq \lambda|u-v|$, for all $(x,\dot x,u)\in\text{dom}(L)$.
\end{itemize}
\end{remark}

\begin{remark}
\
\begin{itemize}
\item [(1)]  For convenience, we denote $(x,p)\in T^*M$, $u\in\mathbb{R}$, by $(x,u,p)\in T^*M\times \mathbb{R}$.
\item [(2)] The assumption (CER) is equivalent to the following statement: for each $R>0$, there exists $K>0$ such that for any $|u|<R$ and $\|p\|>K$, we have $H(x,u,p)>R$. In fact, by (CER), for each $R>0$, there exists $K>0$ such that for $\|p\|>K$, $H(x,0,p)>(1+\lambda)R$. By (LIP), for any $|u|<R$,
    \[H(x,u,p)\geq H(x,0,p)-\lambda |u|>R.\]
    The converse direction is obvious.
    \item [(3)]  dom($L$) is independent of $u$. More precisely, given $(x,\dot{x})\in TM$, if $L(x,u_0,\dot{x})<+\infty$ for a given $u_0\in \R$, then for any $u\in \R$,
        \begin{equation*}
        \begin{aligned}
        L(x,u,\dot{x})&\leq \sup_{p\in T^*_xM}\{\langle \dot x,p\rangle-H(x,u_0,p)\}+\lambda|u-u_0|
        \\ &=L(x,u_0,\dot{x})+\lambda|u-u_0|<+\infty.
        \end{aligned}
        \end{equation*}
        Thus, there holds
        \[\text{dom}(L)=\{(x,\dot{x})\in TM\ |\ L(x,0,\dot{x})<+\infty)\}\times\R.\]
\end{itemize}
\end{remark}

\subsection{An implicit Lax-Oleinik semigroup}

Consider the viscosity solution of the Cauchy problem
\begin{equation}\tag{$CP_H$}\label{C}
  \left\{
   \begin{aligned}
   &\partial_t u(x,t)+H(x,u(x,t),\partial_xu(x,t))=0,\quad (x,t)\in M\times(0,+\infty).
   \\
   &u(x,0)=\varphi(x),\quad x\in M.
   \\
   \end{aligned}
   \right.
\end{equation}
We have the following result.

\begin{theorem}\label{m1}
Assume $H:T^*M\times\mathbb R\rightarrow\mathbb R$ satisfies (C)(CON)(CER)(LIP). The following implicit backward Lax-Oleinik semigroup $T_t^-:C(M)\to C(M)$, via
\begin{equation}\tag{T-}\label{T-}
  T^-_t\varphi(x)=\inf_{\gamma(t)=x} \left\{\varphi(\gamma(0))+\int_0^tL(\gamma(\tau),T^-_\tau\varphi(\gamma(\tau)),\dot{\gamma}(\tau)){d}\tau\right\}
\end{equation}
is well-defined. The infimum  is taken among absolutely continuous curves $\gm:[0,t]\rightarrow M$ with $\gamma(t)=x$. Moreover,
\begin{itemize}
\item [(i)]
if   $\varphi$ is  continuous, then $ u(x,t):=T^-_t\varphi(x)$ represents the unique  continuous viscosity solution of (\ref{C});
 \item [(ii)] if $\varphi$ is Lipschitz continuous, then $u(x,t):=T^-_t\varphi(x)$ is also locally Lipschitz continuous on $M\times [0,+\infty)$.
 \end{itemize}
\end{theorem}

The main  difficulties to prove Theorem \ref{m1} are stated as follows.
\begin{itemize}
\item Compared to the contact HJ equation under the Tonelli conditions, the contact Hamilton flow can not be defined. Consequently, we do not have the compactness of the minimizing orbit set, which plays a crucial role in the  previous work on contact HJ equations (see \cite[Lemma 2.1]{Wa2}).
    \item Compared to the classical HJ equation in less regular cases (see \cite{gen,Fa}), the backward Lax-Oleinik semigroup is implicit defined, which causes $t$-dependence of the Lagrangian. Due to the Lavrentiev phenomenon, it is not direct to prove the Lipschitz continuity of the minimizers of $T^-_t\varphi(x)$ (see \cite{ball} for various counterexamples).
\end{itemize}

\begin{remark}\label{tplus}
Similar to Theorem \ref{m1}, the forward Lax-Oleinik semigroup can be defined as
\begin{equation}\label{T+eq}\tag{T+}
  T^+_t\varphi(x)=\sup_{\gamma(0)=x}\left\{\varphi(\gamma(t))-\int_0^tL(\gamma(\tau),T^+_{t-\tau}\varphi(\gamma(\tau)),\dot{\gamma}(\tau))d\tau\right\}.
\end{equation}
Use the same argument as \cite[Proposition 2.8]{Wa3}, one has $T^+_t\varphi:=-\bar T^-_t(-\varphi)$, where $\bar T^-_t$ denotes the backward Lax-Oleinik semigroup associated to $L(x,-u,-\dot{x})$.
\end{remark}
 By Theorem  \ref{m1}, if the fixed points of $T^-_t$ exist, then they  are viscosity solutions of \begin{equation}\label{E}\tag{$E_H$}
  H(x,u(x),\partial_x u(x))=0.
\end{equation}

Recently, an alternative  variational formulation  was provided in \cite{CCWY,CCJWY,LTW}  in light of G. Herglotz's work \cite{her}, which is related to nonholonomic constraints. By using the Herglotz variational principle, various kinds of representation formulae for the viscosity solutions of (\ref{evll}) were also obtained in  \cite{HCHZ}. For simplicity, we will omit the word ``viscosity" if it is not necessary to be mentioned.
\subsection{An  existence result for the solutions of (\ref{E})}

\begin{remark}\label{Perron}
Let $H:T^*M\times\mathbb R\rightarrow\mathbb R$ satisfy (C)(CER)(LIP). According to the  Perron method \cite{Ish}, if
(\ref{E})
has a Lipschitz subsolution $f$ and a Lipschitz  supersolution $g$ such that $f\leq g$. Then the equation (\ref{E}) admits a Lipschitz viscosity solution.
\end{remark}

In light of  \cite{Ish}, we introduce another necessary and sufficient  condition for (\ref{E}) to admit solutions.
\begin{theorem}\label{S}
Let $H:T^*M\times\mathbb R\rightarrow\mathbb R$ satisfy (C)(CON)(CER)(LIP). The following statements are equivalent:
\begin{itemize}
\item [(1)] (\ref{E}) admits Lipschitz  solutions;
\item [(2)] There exist two continuous functions $\varphi$ and $\psi$ such that $T^-_t\varphi\geq C_1$ and $T^-_t\psi\leq C_2$, where $C_1, C_2$ are constant  independent of $t$ and $x$;

\item [(3)] There exist two continuous functions $\varphi$ and $\psi$, and  $t_1$, $t_2>0$ such that $T^-_{t_1}\varphi\geq \varphi$ and $T^-_{t_2}\psi\leq \psi$.
\end{itemize}
\end{theorem}
If (\ref{E}) admits a solution denoted by $u$, one can take $u$ as the initial function. The statement (2) and (3) hold true obviously. Thus, we only need to show the opposite direction, which will be proved in Section \ref{secS}. The main novelty of Theorem \ref{S} is that  the lower bound of  $T^-_t\varphi$ is not required to  be less than the upper bound of $T^-_t\psi$.

\subsection{The Aubry set}

We denote by $\mathcal S_-$ and $\mathcal S_+$  the set of all backward weak KAM solutions and the set of all forward weak KAM solutions of (\ref{E}) respectively. See Appendix \ref{swv} for their definitions and relations with viscosity solutions. In the discussion below, we need to introduce the following assumption
\begin{itemize}
\item [\textbf{(S):}]   the contact HJ equation (\ref{E}) admits a  solution.
\end{itemize}

\begin{definition}
Let $u_-\in \mathcal{S}_-$, $u_+\in \mathcal{S}_+$. We define the projected Aubry set with respect to $u_-$ by
\begin{equation*}
  \mathcal I_{u_-}:=\{x\in M:\ u_-(x)=\lim_{t\rightarrow +\infty}T^+_tu_-(x)\}.
\end{equation*}
Similarly, we define the projected Aubry set with respect to $u_+$ by
\begin{equation*}
  \mathcal I_{u_+}:=\{x\in M:\ u_+(x)=\lim_{t\rightarrow +\infty}T^-_tu_+(x)\}.
\end{equation*}
In particular, if $u_+(x)=\lim_{t\rightarrow +\infty}T^+_tu_-(x)$ and $u_-(x)=\lim_{t\rightarrow +\infty}T^-_tu_+(x)$, then
\[\mathcal I_{u_-}=\mathcal I_{u_+}.\]
In this special case above, we denote the sets $\mathcal I_{u_-}$ and $\mathcal I_{u_+}$  by $\mathcal I_{(u_-,u_+)}$, following the notation introduced by Fathi \cite{Fa08}.
\end{definition}

\begin{theorem}\label{m2'}
Assume $H:T^*M\times\mathbb R\rightarrow\mathbb R$ satisfies (C)(CON)(CER)(LIP) and (S). Then
\begin{itemize}
\item [(1)] for each $u_-\in\mathcal S_-$, the limit function $U(x):=\lim_{t\rightarrow +\infty}T^+_tu_-(x)$ exists and it is a forward weak KAM solution of (\ref{E});
 \item [(2)] for each $u_+\in\mathcal S_+$, the limit function $V(x):=\lim_{t\rightarrow +\infty}T^-_tu_+(x)$ exists and it is a backward weak KAM solution of (\ref{E});
 \item [(3)]  both  $\mathcal I_{u_-}$ and $\mathcal I_{u_+}$ are nonempty.
 \end{itemize}
\end{theorem}

\subsection{A  comparison result for the solutions of (\ref{E})}
In this part, we are concerned with further properties of the {viscosity solution  for a special class of Hamilton-Jacobi equations   that do not satisfy  the proper condition:
\[
H(x,r,p)\leqslant H(x,s,p)\ \ \text{whenever}\ r\leqslant s.
\]
We assume $H:T^*M\times\mathbb R\rightarrow\mathbb R$ satisfies (C), (CON), (CER), (LIP) and
\begin{itemize}
\item [\textbf{(STD):}] $H(x,u,p)$ is strictly decreasing in $u$.
\end{itemize}
Under the assumptions above, the  solution of $H(x,u,\partial_xu)=0$ is not unique (see e.g., Example (\ref{E0}) below). The following result provides a comparison among different  solutions.
\begin{theorem}\label{four3}
Let $v_1$, $v_2\in \mathcal{S}_-$.
\begin{itemize}
\item [(1)] If $v_1\leq v_2$, then $\emptyset\neq \mathcal I_{v_2}\subseteq \mathcal I_{v_1}$;

\item [(2)] If there is a neighborhood $\mathcal O$ of $\mathcal I_{v_2}$ such that $v_1|_{\mathcal O}\leq v_2|_{\mathcal O}$, then $v_1\leq v_2$ everywhere;

\item [(3)] If $\mathcal I_{v_1}=\mathcal I_{v_2}$ and $v_1|_{\mathcal O}=v_2|_{\mathcal O}$, then $v_1= v_2$ everywhere.
\end{itemize}
\end{theorem}

In order to explain the necessity of the neighbourhood $\mathcal O$, we consider the following example
\begin{equation}\label{E0}\tag{E1}
  -\lambda u(x)+\frac{1}{2}|u'(x)|^2+V(x)=0,\quad x\in\mathbb S\simeq(-1,1],
\end{equation}
where $\mathbb S$ denotes a flat circle with the fundamental domain $(-1,1]$, and $V(x)$ is the restriction of $x^2/2$ on $\mathbb S$. Then \[H(x,u,p)=-\lambda u+\frac{1}{2}|p|^2+V(x)\]  is Lipschitz continuous. Assume $\lambda>2$, a direct calculation shows that  there are two  viscosity solutions given by
\begin{equation*}
  u_1(x)=\frac{\lambda+\sqrt{\lambda^2-4}}{2}V(x),\quad u_2(x)=\frac{\lambda-\sqrt{\lambda^2-4}}{2}V(x).
\end{equation*}
It can be shown that  $\mathcal I_{u_1}=\mathcal I_{u_2}=\{0\}$, although $u_1\neq u_2$ on $\mathbb S$.
A detailed analysis of Example (\ref{E0}) is given by Section \ref{exxxam} below.

\bigskip

The rest of this paper is organized as follows. In Section \ref{main1}, we prove Theorem  \ref{m1}. To achieve that, we need some technical lemmas whose proofs are given in Appendix \ref{apcc} and \ref{apbb}. Theorem  \ref{S}, Theorem \ref{m2'} and Theorem \ref{four3} are proved in Section \ref{secS}, Section \ref{pm2'} and Section \ref{pm27} successively.  In addition, we give some basic results on the existence and regularity of the minimizers of one dimensional variational problems in Appendix \ref{preli}, and we also provide some basic properties of weak KAM solution and viscosity solution in Appendix \ref{swv} for the reader's convenience.

\bigskip

We list notations  in the present paper:
\begin{itemize}
\item [$\centerdot$] $\textrm{diam}(M)$ denotes the diameter of $M$;

\item [$\centerdot$] $d(x,y)$ denotes the distance between $x$ and $y$ induced by the Riemannian metric $g$ on $M$;

\item [$\centerdot$] $\|\cdot\|$ denotes the norms induced by $g$ on both tangent and cotangent spaces of $M$;

\item[{$\centerdot$}] $B(v,r)$ stands for the open norm ball on $T_xM$ centered at $v\in T_xM$ with radius $r$, and $\bar B(v,r)$ stands for its closure;

\item [$\centerdot$] $C(M)$ stands for the space of continuous functions on $M$;

\item [$\centerdot$] $\text{Lip}(M)$ stands for the space of Lipschitz continuous functions on $M$;

\item [$\centerdot$] $\|\cdot\|_\infty$ stands for the supremum norm of the vector-valued function on its domain.
\end{itemize}

\section{An implicit Lax-Oleinik semigroup}\label{main1}
In this part, we are devoted to proving Theorem \ref{m1}. It is needed to show
\begin{itemize}
\item [($\ast$)]
if  the initial condition $\varphi$ is Lipschitz continuous, then $u(x,t):=T^-_t\varphi(x)$ is the Lipschitz  solution of (\ref{C});
 \item [($\ast\ast$)]if $\varphi$ is  continuous, then $u(x,t):=T^-_t\varphi(x)$ is the continuous  solution of (\ref{C}).
\end{itemize}

\subsection{On Item ($\ast$): the Lipschitz initial condition}
For the reader's convenience, we give a sketch of proof of Item  ($\ast$) as follows.
\begin{itemize}
\item [(1)] Lemma \ref{Bol} is proved in Appendix \ref{apcc}. Some background knowledge is given in Section \ref{A.11}.
 \item [(2)] Under the following \textbf{Condition 1}:

 {\it $u_k$ defined in (\ref{k}) is continuous on  $M\times [0,T]$ for each $k\in\mathbb N_+$,}

  \noindent we prove Lemma \ref{begin}(i) based on Lemma \ref{Bol}.
   \item [(3)] Under the following \textbf{Condition 2}:

 {\it $u_{k}$ is locally Lipschitz  on $M\times (0,T]$ for each $k\in \mathbb{N}_+$, and it is the  solution of (\ref{uk}) below,}

  \noindent we prove Lemma \ref{begin}(ii) by Lemma \ref{begin}(i).
  \item [(4)]After a {\it superlinear} modification, we have Lemma \ref{3.1} whose proof is provided in Appendix \ref{apbb}. The proof  needs some ingredients given in Section \ref{A.2}.
  \item [(5)] \textbf{Condition 2} in Lemma \ref{begin}(ii) can be verified by Lemma  \ref{3.1} under the {\it superlinear} condition. By Lemma \ref{begin}(ii),  Item  ($\ast$) holds under the {\it superlinear} condition.
  \item [(6)]  Under the {\it coercive} condition, we prove Lemma \ref{ukLip} based on Item  ($\ast$) under the {\it superlinear} condition.
  \item [(7)] Under the {\it coercive} condition, \textbf{Condition 2} in Lemma \ref{begin}(ii) can be verified by Lemma  \ref{ukLip}. By Lemma \ref{begin}(ii),  Item  ($\ast$) holds under the {\it coercive} condition.
\end{itemize}

\vspace{1em}

The following is a detailed proof of Item  ($\ast$).
\begin{lemma}\label{Bol}
Fix $T>0$. Given $\varphi\in C(M,\mathbb R)$, $v\in C(M\times [0,T],\mathbb R)$ and $t\in [0,T]$, the functional
\begin{equation*}
  \mathbb L^t(\gamma):=\varphi(\gamma(0))+\int_0^t L(\gamma(s),v(\gamma(s),s),\dot{\gamma}(s))ds
\end{equation*}
reaches its infimum in the class of curves
\begin{equation*}
  X_t(x)=\{\gamma\in W^{1,1}([0,t],M):\ \gamma(t)=x\}.
\end{equation*}
\end{lemma}

\begin{lemma}\label{begin}
Fix $T>0$ and $\varphi\in C(M)$. For $k\in\mathbb N_+$ and $t\in (0,T]$, consider the following iteration procedure
\begin{equation}\label{k}
  u_k(x,t):=\inf_{\gamma(t)=x}\left\{\varphi(\gamma(0))+\int_0^tL(\gamma(\tau),u_{k-1}(\gamma(\tau),\tau),\dot{\gamma}(\tau)){d}\tau\right\},
\end{equation}
where $u_0(x,t):=\varphi(x)$.
\begin{itemize}
\item [(i)]
If $u_k$ is continuous on  $M\times [0,T]$ for each $k\in\mathbb N_+$, then  $\{u_k(x,t)\}_{k\in \mathbb{N}}$ converges uniformly  to $u(x,t):=T^-_t\varphi(x)$ for all $(x,t)\in M\times[0,T]$, where the semigroup $T_t^-:C(M)\to C(M)$ is formulated as (\ref{T-}).
 \item [(ii)] Let $\varphi\in \text{Lip}(M)$. If  $u_{k}$ is locally Lipschitz  on $M\times (0,T]$ for each $k\in \mathbb{N}_+$, and it is the  solution of
 \begin{equation}\label{uk}
  \left\{
   \begin{aligned}
   &\partial_t u(x,t)+H(x,u_{k-1}(x,t),\partial_xu(x,t))=0,
   \\
   &u(x,0)=\varphi(x),
   \\
   \end{aligned}
   \right.
\end{equation}
 then $u_k$ is Lipschitz on $M\times [0,T]$, and its Lipschitz constant depends only on $\sup_{k\in \mathbb N}\|u_k\|_\infty$ and $\|\partial_x\varphi\|_\infty$. Moreover, the limit function $u(x,t):=T_t^-\varphi$ is Lipschitz.

 \end{itemize}
\end{lemma}
\begin{proof}
{\it Item (i):} By Lemma \ref{Bol}, the minimizers of each $u_k$ exist. First, let $\gm_1:[0,t]\to M$ be a minimizer of $u_1(x,t)$, then
\begin{equation*}
\begin{aligned}
  u_2(x,t)-u_1(x,t)&\leq \int_0^t\bigg[L(\gm_1(s),u_1(\gm_1(s),s),\dot\gm_1(s))-L(\gm_1(s),\varphi(\gm_1(s)),\dot\gm_1(s))\bigg]ds
  \\ &\leq \lambda \|u_1-\varphi\|_\infty t.
\end{aligned}
\end{equation*}
Exchanging the positions of $u_2$ and $u_1$, we obtain
\[|u_2(x,t)-u_1(x,t)|\leq \lambda\|u_1-\varphi\|_\infty t.\]
Let $\gm_2:[0,t]\to M$ be a minimizer of $u_2(x,t)$, then
\begin{equation*}
\begin{aligned}
  u_3(x,t)-u_2(x,t)&\leq \int_0^t\bigg[L(\gm_2(s),u_2(\gm_2(s),s),\dot\gm_2(s))-L(\gm_2(s),u_1(\gm_2(s),s),\dot\gm_2(s))\bigg]ds
  \\ &\leq \lambda\int_0^t|u_2(\gm_2(s),s)-u_1(\gm_2(s),s)|ds
  \leq \lambda \int_0^t\lambda\|u_1-\varphi\|_\infty sds
  \\ &=\frac{(\lambda t)^2}{2}\|u_1-\varphi\|_\infty.
\end{aligned}
\end{equation*}
Exchanging the positions of $u_3$ and $u_2$, we obtain
\[|u_2(x,t)-u_1(x,t)|\leq \frac{(\lambda t)^2}{2}\|u_1-\varphi\|_\infty.\]
Continuing the above procedure, we obtain
\[\|u_{j+1}-u_{j}\|_\infty \leq \frac{(\lambda T)^j}{j!}\|u_1-\varphi\|_\infty,\]
for each $j\in\mathbb N$. Thus,
\begin{equation*}
  \|u_k-\varphi\|_\infty\leq \sum_{j=0}^{k-1}\|u_{j+1}-u_j\|_\infty\leq \sum_{j=0}^{k-1}\frac{(\lambda T)^j}{j!}\|u_1-\varphi\|_\infty\leq e^{\lambda T}\|u_1-\varphi\|_\infty,\quad \forall k\in\mathbb N_+.
\end{equation*}
For $k_1>k_2$, we have
\[\|u_{k_1}-u_{k_2}\|_\infty\leq \frac{(\lambda T)^{k_2}}{k_2!}\|u_{k_1-k_2}-\varphi\|_\infty\leq \frac{(\lambda T)^{k_2}}{k_2!}e^{\lambda T}\|u_1-\varphi\|_\infty.\]
Since $(\lambda T)^k/k!$ converges to zero as $k\to \infty$, the right hand side can be arbitrarily small when $k_2$ is large enough. Therefore, the sequence $\{u_k(x,t)\}_{k\in\mathbb N}$ is a Cauchy sequence in the Banach space $(C(M\times[0,T]),\|\cdot\|_\infty)$. Then $\{u_k(x,t)\}_{k\in\mathbb N}$ converges uniformly to a continuous function $u(x,t)$. Define $A_\varphi:C(M\times[0,T])\to C(M\times[0,T])$ via
\begin{equation*}
  \mathcal A_\varphi[u](x,t):=\inf_{\gamma(t)=x}\left\{\varphi(\gamma(0))+\int_0^tL(\gamma(\tau),u(\gamma(\tau),\tau),\dot{\gamma}(\tau)){d}\tau\right\}.
\end{equation*}
Then the limit function $u(x,t)$ satisfies
\[\|\mathcal A_\varphi[u]-u\|_\infty\leq \|\mathcal A_\varphi[u]-u_k\|_\infty+\|u_k-u\|_\infty\leq \lambda T\|u-u_{k-1}\|_\infty+\|u_k-u\|_\infty.\]
Setting $k\rightarrow+\infty$ we conclude that $u(x,t)$ is the unique fixed point of $\mathcal A_\varphi$. Namely, $u(x,t):=T_t^-\varphi$. The semigroup property of $T_t^-$ can be verified by a similar argument as \cite[Proposition 3.3]{JWY}.

\vspace{1ex}

\noindent {\it Item (ii):} By Item (i), $\{u_k(x,t)\}_{k\in \mathbb{N}}$ converges uniformly, then \[\sup_{k\in \mathbb N}\|u_k(x,t)\|_\infty<+\infty.\] Define
\[K_1:=\max \{|H(x,u,p)|:\ x\in M,\ |u|\leq \sup_{k\in \mathbb N}\|u_k(x,t)\|_\infty,\ \|p\|\leq \|\partial_x\varphi(x)\|_\infty\},\]
and
\[K_2:=\min \{H(x,u,p):\ (x,p)\in T^*M,\ |u|\leq \sup_{k\in \mathbb N}\|u_k(x,t)\|_\infty\}.\]
We will prove by induction
\begin{equation}\label{partialt}
  \|\partial_t u_k(\cdot,t)\|_\infty \leq \max\{K_1e^{\lambda t},|K_2|\}
\end{equation}
for each $k\in\mathbb N_+$.

First, let us consider the case $k=1$. Define
\[K_0:=\max \{|H(x,u,p)|:\ x\in M,\ |u|\leq \|\varphi(x)\|_\infty,\ \|p\|\leq \|\partial_x\varphi(x)\|_\infty\}.\]
For any $h>0$, we define
\begin{equation}\label{barw1}
  w(x,t):=\left\{
   \begin{aligned}
   &\varphi(x)-K_0t,\quad t\leq h,
   \\
   &u_k(x,t-h)-K_0h,\quad t>h.
   \\
   \end{aligned}
   \right.
\end{equation}
First, the Lipschitz function $(x,t)\mapsto\varphi(x)-K_0t$ satisfies
\begin{equation}\label{k=1}
  \partial_t u+H(x,\varphi(x),\partial_x u)\leq 0
\end{equation}
almost everywhere. According to Proposition \ref{aesub}, $w(x,t)$ is a subsolution of (\ref{k=1}) for $t\leq h$. Also, $w(x,t)$ satisfies (\ref{k=1}) for $t>h$. Thus, $w(x,t)$ is a continuous subsolution of (\ref{k=1}) with $w(x,0)=\varphi(x)$. By the comparison result (see, e.g., \cite[Theorem 5.1]{Intr}), since $(x,t)\mapsto\varphi(x)-K_0t$ is Lipschitz in $x$, we have  \[w(x,h)=\varphi(x)-K_0h\leq u_1(x,h).\] Note that $u_1(x,t)$ is Lipschitz on $M\times[h,T]$, we have
\begin{equation*}
  u_1(x,t)-K_0h=w(x,t+h)\leq u_1(x,t+h),\quad \forall t\geq 0,\ h>0.
\end{equation*}
Let $h\to 0^+$. We have $\partial_t u_1(x,t)\geq -K_0\geq -K_1$. We also have
\[\partial_t u_1(x,t)=-H(x,\varphi(x),\partial_x u_1(x,t))\leq |K_2|.\]
Thus, (\ref{partialt}) holds for $k=1$.

Now assume (\ref{partialt}) holds for $k-1$. For any $h>0$, we define
\begin{equation}\label{barw2}
  \bar w(x,t):=\left\{
   \begin{aligned}
   &\varphi(x)-K_1e^{\lambda h}t,\quad t\leq h,
   \\
   &u_k(x,t-h)-K_1he^{\lambda t},\quad t>h.
   \\
   \end{aligned}
   \right.
\end{equation}
First, the Lipschitz function $(x,t)\mapsto\varphi(x)-K_1e^{\lambda h}t$ satisfies \[\partial_t u+H(x,{u_{k-1}}(x,t),\partial_x u)\leq 0\] almost everywhere. According to Proposition \ref{aesub}, $\bar w(x,t)$ is a subsolution of (\ref{uk}) for $t\leq h$. For $t>h$, we have
\begin{equation*}
\begin{aligned}
  &\partial_t \bar w(x,t)+H(x,u_{k-1}(x,t),\partial_x \bar w(x,t))
  \\ &=\partial_t u_k(x,t-h)-K_1h\lambda e^{\lambda t}+H(x,u_{k-1}(x,t),\partial_x u_k(x,t-h))
  \\ &\leq \partial_t u_k(x,t-h)-\lambda \sup_{s\in[t-h,t]}\|\partial_t u_{k-1}(\cdot,s)\|_\infty h+H(x,u_{k-1}(x,t),\partial_x u_k(x,t-h))
  \\ &\leq \partial_t u_k(x,t-h)+H(x,u_{k-1}(x,t-h),\partial_x u_k(x,t-h))=0.
\end{aligned}
\end{equation*}
By the comparison result, since $(x,t)\mapsto\varphi(x)-K_1e^{\lambda h}t$ is Lipschitz in $x$, we have  \[\bar w(x,h)=\varphi(x)-K_1he^{\lambda h}\leq u_k(x,h).\] Note that $u_k(x,t)$ is Lipschitz on $M\times[h,T]$, we have
\begin{equation*}
  u_k(x,t)-K_1he^{\lambda (t+h)}=\bar w(x,t+h)\leq u_k(x,t+h),\quad \forall t\geq 0,\ h>0.
\end{equation*}
Let $h\to 0^+$. We have $\partial_t u_k(x,t)\geq -K_1e^{\lambda t}$. We also have
\[\partial_t u_k(x,t)=-H(x,u_{k-1}(x,t),\partial_x u_k(x,t))\leq |K_2|.\]
We conclude that (\ref{partialt}) holds for $k$. Plugging them into (\ref{uk}), one obtain
\begin{equation*}
  H(x,0,\partial_xu_k(x,t))\leq \max\{K_1e^{\lambda T},|K_2|\}+\lambda \|{u_{k-1}}(x,t)\|_\infty.
\end{equation*}
Thus $\|\partial_x u_k(x,t)\|_\infty$ is bounded on $M\times[0,T]$ by (CER). It means $u_k(x,t)$ is Lipschitz on $M\times[0,T]$, and the Lipschitz constant only depends on $\sup_{k\in \mathbb N}\|u_k(x,t)\|_\infty$ and $\|\partial_x\varphi(x)\|_\infty$. Moreover,  $\{u_k(x,t)\}_{k\in\mathbb N}$ is equi-Lipschitz with respect to $k$. It follows that the limit function $u(x,t):=T_t^-\varphi$ is Lipschitz.
\end{proof}

According to Lemma \ref{begin} (ii), the key point for the proof of Item ($\ast$) is to show  for each $k\in \mathbb{N}$, $u_k(x,t)$ defined by (\ref{k}) is locally Lipschitz, and solves (\ref{uk}) in the viscosity sense. 
We divide the remaining proof into two steps. In Step 1, we prove Item ($\ast$) for the Hamiltonian $H(x,u,p)$ depending on $p$ superlinearly. In Step 2, the superlinearity is relaxed to (CER).

\subsubsection{Step 1: Proof under the superlinear condition}\label{H3'}

In this part, we assume the Hamiltonian $H:T^*M\times\mathbb R\rightarrow\mathbb R$ satisfies (C)(CON)(LIP) and the following.
\begin{itemize}
\item [\textbf{(SL):}]
For every $(x,u)\in M\times \mathbb{R}$, $H(x,u,p)$ is superlinear in $p$, i.e. there exists a function $\Theta:[0+\infty)\rightarrow\mathbb R$ satisfying
\begin{equation*}
  \lim_{r \rightarrow+\infty}\frac{\Theta(r)}{r}=+\infty,\quad \textrm{and}\quad H(x,u,p)\geq \Theta(\|p\|)\quad \textrm{for\ every}\ (x,p,u)\in T^*M\times\mathbb R.
\end{equation*}
\end{itemize}

The corresponding Lagrangian satisfies (CON)(LIP) and
\begin{itemize}
\item [\textbf{(C):}] $L(x,u,\dot x)$ is continuous;

\item [\textbf{(SL):}] For every $(x,u)\in M\times \mathbb{R}$, $L(x,u,\dot x)$ is superlinear in $\dot x$, i.e., there exists a function $\Theta:[0+\infty)\rightarrow\mathbb R$ satisfying
\begin{equation*}
  \lim_{r \rightarrow+\infty}\frac{\Theta(r)}{r}=+\infty,\quad \textrm{and}\quad L(x,u,\dot x)\geq \Theta(\|\dot x\|)\quad \textrm{for\ every}\ (x,\dot x,u)\in TM\times\mathbb R.
\end{equation*}
\end{itemize}
At the beginning, we need some technical results.
\begin{lemma}\label{3.1}
Given $T>0$ and  $\varphi\in C(M)$, if $v(x,t)$ is a Lipschitz function on $M\times[0,T]$, then
\begin{itemize}
\item [(1)] for any $(x,t)\in M\times[0,T]$, the minimizers of
\begin{equation}\label{u}
  u(x,t):=\inf_{\gamma(t)=x}\left\{\varphi(\gamma(0))+\int_0^tL(\gamma(\tau),v(\gamma(\tau),\tau),\dot{\gamma}(\tau)){d}\tau\right\}
\end{equation}
are Lipschitz. For any $r>0$, if $d(x,x')\leq r$ and $|t-t'|\leq r/2$, where $t\geq r> 0$, then the Lipschitz constants of the minimizers of $u(x',t')$ only depend on $(x,t)$ and $r$.
\item [(2)]  the value function $u(x,t)$ defined in (\ref{u}) is locally Lipschitz on $M\times(0,T]$.
\item [(3)] $u(x,t)$  is also the viscosity solution of
\begin{equation}\label{cu}
  \left\{
   \begin{aligned}
   &\partial_t u(x,t)+H(x,v(x,t),\partial_xu(x,t))=0,
   \\
   &u(x,0)=\varphi(x).
   \\
   \end{aligned}
   \right.
\end{equation}
on $M\times[0,T]$.
\end{itemize}
\end{lemma}

Based on Lemma \ref{3.1}, we verify Item ($\ast$) under the assumption (SL). In fact, let $u_0:=\varphi\in \text{Lip}(M)$ in the iteration procedure given by (\ref{k}). By Lemma \ref{begin} (i), $u_k(x,t)$ converges uniformly  to $u(x,t):=T_t^-\varphi(x)$ on $M\times[0,T]$. By Lemma \ref{3.1} (2) and (3), $u_1(x,t)$ satisfies the assumptions in Lemma \ref{begin} (ii), by which  $u_1$ is Lipschitz on $M\times [0,T]$. Repeating the argument, one can obtain that $u_k$ is the  Lipschitz  solution of (\ref{uk}).  By Lemma \ref{begin} (ii), the Lipschitz constant of $u_k(x,t)$ is uniform with respect to $k$ on $M\times[0,T]$. Since $H_k(t,x,p):=H(x,u_k(x,t),p)$ converges uniformly on compact subsets of $\mathbb R\times T^*M$, and $u_k(x,t)$ converges uniformly on $M\times[0,T]$, then the backward semigroup $u(x,t):=T^-_t\varphi(x)$, as the limit of $u_k(x,t)$, is the  Lipschitz  solution of (\ref{C}) by the stability of viscosity solutions.

\subsubsection{Step 2: Relaxed to the coercive condition}\label{H3}

In this part, we assume the Hamiltonian $H:T^*M\times\mathbb R\rightarrow\mathbb R$ satisfies (C)(CON)(CER)(LIP). By Lemma \ref{Bol}, one has the existence of the minimizers. In order to obtain the Lispchitz regularity of $u_k$ in (\ref{k}).
We make a modification:
\begin{equation*}
  H_n(x,u,p):=H(x,u,p)+\max\{\|p\|^2-n^2,0\},\quad n\in\mathbb N.
\end{equation*}
It is clear that $H_n$ is superlinear in $p$. The sequence $H_n$ is decreasing, and converges uniformly to $H$ on compact subsets of $T^*M\times\mathbb R$. The sequence of the corresponding Lagrangians $\{L_n\}_{n\in \mathbb{N}}$ is increasing, and converges to $L$ pointwisely. Denote by $u_{n,k}(x,t)$ the  solution of (\ref{uk}) with $H$ replaced by $H_n$.


\begin{lemma}\label{ukLip}Let $H$ satisfy (C)(CON)(CER)(LIP). Let  $L$ be the Lagrangian associated to $H$. Given $\varphi\in \text{Lip}(M)$,
for each $k\in\mathbb N$, the function $u_k(x,t)$ defined by (\ref{k}) is the  Lipschitz  solution of (\ref{uk}).
\end{lemma}
\begin{proof}
Given $n\in \mathbb{N}$, let
\begin{equation}\label{k2.4x}
  u_{n,k}(x,t):=\inf_{\gamma(t)=x}\left\{\varphi(\gamma(0))+\int_0^tL_n(\gamma(\tau),u_{n,k-1}(\gamma(\tau),\tau),\dot{\gamma}(\tau)){d}\tau\right\},
\end{equation}
with $u_{n,0}:=\varphi\in \text{Lip}(M)$.
We first prove the following assertion for each $k\in\mathbb N$ by induction.
\begin{itemize}
\item [\textbf{A[k]}:] Fix $k\in \mathbb{N}$. The sequence $\{u_{n,k}(x,t)\}_{n\in \mathbb{N}}$ is uniformly bounded and equi-Lipschitz with respect to $n$, and converges uniformly to $u_k(x,t)$ on $M\times[0,T]$. Moreover, the limit function $u_k(x,t)$ is Lipschitz.
\end{itemize}

We first prove that the assertion A[1] holds. By definition of $H_n$, for $n\geq \|\partial_x\varphi(x)\|_\infty$,
\begin{equation*}
  K^n_0:=\max\{|H_n(x,u,p)|:\ x\in M,\ |u|\leq \|\varphi(x)\|_\infty,\ \|p\|\leq \|\partial_x\varphi(x)\|_\infty\}
\end{equation*}
is always equal to
\begin{equation*}
  K_0=\max\{|H(x,u,p)|:\ x\in M,\ |u|\leq \|\varphi(x)\|_\infty,\ \|p\|\leq \|\partial_x\varphi(x)\|_\infty\}.
\end{equation*}
Note that $u_{n,1}(x,t)$ is the solution of
\begin{equation}\label{n1}
  \partial_t u+H_n(x,\varphi(x),\partial_x u)=0.
\end{equation}
Similar to the proof of (\ref{partialt}) with $k=1$, we have \[\partial_tu_{n,1}(x,t)\geq -K^n_0.\]
Combining (\ref{n1}) and the definition of $H_n$, we have
\begin{equation*}
  H(x,0,\partial_xu_{n,1}(x,t))\leq H_n(x,0,\partial_xu_{n,1}(x,t))
  \leq K_0+\lambda \|\varphi(x)\|_\infty,\quad \forall n\geq \|\partial_x\varphi(x)\|_\infty.
\end{equation*}
Therefore, $\{u_{n,1}(x,t)\}_{n\in \mathbb{N}}$ is equi-Lipschitz. Note that \[u_{n,1}(x,0)=\varphi(x).\] It follows that  $\{u_{n,1}(x,t)\}_{n\in \mathbb{N}}$ is uniformly bounded, so it has a converging subsequence. According to Lemma \ref{inf}, $u_{n,1}(x,t)$ converges to $u_1(x,t)$ pointwisely. It follows that  \[\lim_{n\to+\infty}u_{n,1}(x,t)=u_{1}(x,t),\quad \text{uniformly\ on\ }M\times[0,T],\] which implies that $u_1(x,t)$ is Lipschitz.

Now assume that the assertion A[k-1] holds. Then $u_{k-1}(x,t)$ is Lipschitz, and $l_{k-1}:=\sup_{n\in\mathbb N} \|u_{n,k-1}(x,t)\|_\infty$ is finite. We will prove A[k] from A[k-1]. 
First, plugging the continuous function $u_{k-1}(x,t)$ into (\ref{k}) and by Lemma \ref{Bol}, the minimizers of $u_k(x,t)$
exist in the class of absolutely continuous curves. By definition of $H_n$, for $n\geq \|\partial_x\varphi(x)\|_\infty$,
\begin{equation*}
  K_n:=\max\{|H_n(x,u,p)|:\ x\in M,\ |u|\leq l_{k-1},\ \|p\|\leq \|\partial_x\varphi(x)\|_\infty\}
\end{equation*}
 is always equal to
\begin{equation*}
  K:=\max\{|H(x,u,p)|:\ x\in M,\ |u|\leq l_{k-1},\ \|p\|\leq \|\partial_x\varphi(x)\|_\infty\}.
\end{equation*}
Note that $u_{n,k}$ is the solution of
\begin{equation}\label{nk}
  \partial_t u+H_n(x,u_{n,k-1}(x,t),\partial_x u)=0.
\end{equation}
Similar to the proof of (\ref{partialt}), we have \[\partial_tu_{n,k}(x,t)\geq -K_n-\lambda\|\partial_t u_{n,k-1}\|_\infty T.\]
Combining (\ref{nk}) and the definition of $H_n$, we have
\begin{equation*}
\begin{aligned}
  H(x,0,\partial_xu_{n,k}(x,t))&\leq H_n(x,0,\partial_xu_{n,k}(x,t))
  \\ &\leq K+\lambda\|\partial_t u_{n,k-1}\|_\infty T+\lambda l_{k-1},\quad \forall n\geq \|\partial_x\varphi(x)\|_\infty.
\end{aligned}
\end{equation*}
Therefore, $\{u_{n,k}(x,t)\}_{n\in \mathbb{N}}$ is equi-Lipschitz. Note that \[u_{n,k}(x,0)=\varphi(x).\] It follows that  $\{u_{n,k}(x,t)\}_{n\in \mathbb{N}}$ is uniformly bounded, so it has a converging subsequence. We have to show that all converging subsequences have the same limit function $u_k$. In fact, according to Lemma \ref{inf}, the value function
\[\bar u_{n,k}(x,t):=\inf_{\gamma(t)=x}\left\{\varphi(\gamma(0))+\int_0^tL_n(\gamma(\tau),u_{k-1}(\gamma(\tau),\tau),\dot{\gamma}(\tau)){d}\tau\right\}\] converges to $u_{k}(x,t)$ pointwisely. Taking a minimizer $\gamma$ of $u_{n,k}(x,t)$, we have
\begin{equation*}
\begin{aligned}
  \bar u_{n,k}(x,t)-u_{n,k}(x,t)\leq & \varphi(\gamma(0))+\int_0^tL_n(\gamma(\tau),u_{k-1}(\gamma(\tau),\tau),\dot{\gamma}(\tau)){d}\tau
  \\ &-\varphi(\gamma(0))+\int_0^tL_n(\gamma(\tau),u_{n,k-1}(\gamma(\tau),\tau),\dot{\gamma}(\tau)){d}\tau
  \\ \leq &\lambda\|u_{k-1}(x,t)-u_{n,k-1}(x,t)\|_\infty T.
\end{aligned}
\end{equation*}
Exchanging the role of $\bar u_{n,k}(x,t)$ and $u_{n,k}(x,t)$, we have $\|\bar u_{n,k}(x,t)-u_{n,k}(x,t)\|_\infty\rightarrow 0$ as $n\to\infty$. It follows that   \[\lim_{n\to+\infty}u_{n,k}(x,t)=u_{k}(x,t),\quad \text{uniformly\ on\ }M\times[0,T],\]which implies that $u_k(x,t)$ is Lipschitz.  Note that the Lipschitz constant may depend on $k$. Thus, the assertion A[k] holds.

Since $H_n$ converges uniformly  to $H$ on compact subsets of $T^*M\times\mathbb R$, and $u_{n,k}(x,t)$ converges  uniformly  to $u_k(x,t)$ on $M\times[0,T]$, by the stability of the viscosity solutions, we conclude that $u_k(x,t)$ is the  Lipschitz  solution of (\ref{uk}).
\end{proof}

By Lemma \ref{begin} (i),  $u_k(x,t)$  converges uniformly to $u(x,t)$ on $M\times[0,T]$. Moreover, the value \[\sup_{k\in \mathbb N}\|u_k(x,t)\|_\infty\] is finite. Since $\varphi\in \text{Lip}(M)$, then $\|\partial_x\varphi\|_\infty$ is also finite. By Lemma \ref{begin} (ii),   $\{u_k(x,t)\}_{k\in \mathbb{N}}$ is equi-Lipschitz. Therefore the limit function $u(x,t):=T^-_t\varphi(x)$ is the  Lipschitz  solution of (\ref{C}). So far, Theorem  \ref{m1} has been proved when $\varphi$ is Lipschitz.

\subsection{On Item ($\ast\ast$): the continuous initial conditions} At the beginning, we need to show $T_t^-\varphi$ is well-defined for each $\varphi\in C(M)$. By Lemma \ref{begin}, it suffices to
prove that given $T>0$ and $\varphi\in C(M)$,  $u_k$ defined in (\ref{k}) is  continuous  on $M\times[0,T]$.

In fact, for any $\varphi\in C(M)$, there exists a sequence of Lipschitz functions $\{\varphi_m\}_{m\in \mathbb{N}}$  converging uniformly to $\varphi$. We have already proved in Lemma \ref{ukLip} that for the initial function $\varphi_m$, the solution of (\ref{uk}), denoted by $u^m_k(x,t)$, is Lipschitz. We then proceed by induction in the following.
By definition, $u^m_0$  converges uniformly to $u_0$. Assume $u^m_{k-1}$  converges uniformly to $u_{k-1}$, then $u_{k-1}$ is continuous. By Lemma \ref{Bol}(i), $u_k(x,t)$ admits a minimizer $\gamma$. By definition, we have
\begin{equation*}
  u^m_k(x,t)-u_k(x,t)
  \leq \varphi_m(\gamma(0))-\varphi(\gamma(0))+\lambda\|u^m_{k-1}(x,t)-u_{k-1}(x,t)\|_\infty T.
\end{equation*}
Exchanging the roles of $u^m_k(x,t)$ and $u_k(x,t)$, we obtain $\|u^m_k-u_k\|_\infty\rightarrow 0$ as $m\to\infty$. Therefore, $u_k$ defined in (\ref{k}) is  continuous  on $M\times[0,T]$.


By Lemma \ref{begin} (i),   $u_k(x,t)$ converges uniformly  to $u(x,t):=T_t^-\varphi(x)$. It follows that  $u(x,t)$ is continuous. It remains to verify that  $u(x,t)$ is the solution of (\ref{C}).

 We have proved in Item ($\ast$) that for $\varphi\in \text{Lip}(M)$, $T^-_t\varphi(x)$ is the  Lipschitz  solution of (\ref{C}). We assert for any $\varphi$ and $\psi\in C(M)$, \begin{equation}\label{tno}\|T^-_t\varphi-T^-_t\psi\|_\infty\leq e^{\lambda t}\|\varphi-\psi\|_\infty.\end{equation}
If the assertion is true, then
 for $t\in [0,T]$, $T^-_t\varphi_m$ converges uniformly  to $T^-_t\varphi$ as $m\to \infty$. According to the stability of viscosity solutions, we conclude that $T^-_t\varphi$ is the solution of (\ref{C}) under the initial condition $u(x,0)=\varphi(x)$. The uniqueness of the  solution of (\ref{C}) is guaranteed by the comparison theorem (see \cite[Theorem 2.1]{Ish6}). The assertion (\ref{tno}) above will be verified in Proposition \ref{psg} below.

\section{An  existence result for the solutions of (\ref{E})}\label{secS}

In order to prove Theorem  \ref{S}, we collect two basic properties of the backward and forward Lax-Oleinik semigroups in the following.

\begin{proposition}\label{psg}
\
\begin{itemize}
\item [(1)] For any $\varphi$ and $\psi \in C(M)$, if $\varphi(x)<\psi(x)$ for all $x\in M$, then $T^-_t\varphi(x)< T^-_t\psi(x)$ and $T^+_t\varphi(x)< T^+_t\psi(x)$ for all $(x,t)\in M\times(0,+\infty)$.

\item [(2)] For any $\varphi$ and $\psi \in C(M)$, then $\|T^-_t\varphi-T^-_t\psi\|_\infty\leq e^{\lambda t}\|\varphi-\psi\|_\infty$ and $\|T^+_t\varphi-T^+_t\psi\|_\infty\leq e^{\lambda t}\|\varphi-\psi\|_\infty$ for all $t>0$.
\end{itemize}
\end{proposition}

\begin{proof} We first prove Item (1).  Assume that there exists $(x,t)\in M\times[0,+\infty)$ such that $T^-_t\varphi(x)\geq T^-_t\psi(x)$. Let $\gamma:[0,t]\rightarrow M$ be a minimizer of $T^-_t\psi(x)$ with $\gamma(t)=x$. Define
\begin{equation*}
  F(s)=T^-_s\psi(\gamma(s))-T^-_s\varphi(\gamma(s)),\quad s\in [0,t].
\end{equation*}
Then $F$ is a continuous function defined on $[0,t]$, and $F(0)> 0$. By assumption we have $F(t)\leq 0$. Then there is $s_0\in[0,t)$ such that $F(s_0)=0$ and $F(s)>0$ for all $s\in [0,s_0)$. Since $\gamma$ is a minimizer of $T^-_t\varphi_2(x)$, we have
\begin{equation*}
  T^-_{s_0}\psi(\gamma(s_0))=T^-_{s}\psi(\gamma(s))+\int_{s}^{s_0}L(\gamma(\tau),T^-_\tau\psi(\gamma(\tau)),\dot{\gamma}(\tau))d\tau,
\end{equation*}
and
\begin{equation*}
  T^-_{s_0}\varphi(\gamma(s_0))\leq T^-_{s}\varphi(\gamma(s))+\int_{s}^{s_0}L(\gamma(\tau),T^-_\tau\varphi(\gamma(\tau)),\dot{\gamma}(\tau))d\tau,
\end{equation*}
which implies $F(s_0)\geq F(s)-\lambda \int_{s}^{s_0}F(\tau)d\tau$. Here $F(s_0)=0$, thus
\begin{equation*}
  F(s)\leq \lambda \int_{s}^{s_0}F(\tau)d\tau.
\end{equation*}
By the Gronwall inequality, we conclude $F(s)\equiv 0$ for all $s\in [0,s_0)$, which contradicts $F(0)>0$.

\vspace{1ex}

Next, we prove Item (2). For a given $x\in M$ and $t>0$, if $T^-_t\varphi(x)=T^-_t\psi(x)$, then the proof is completed. Without loss of generality, we consider $T^-_t\varphi(x)>T^-_t\psi(x)$. Let $\gamma:[0,t]\to M$ be a minimizer of $T^-_t\psi(x)$, define
\begin{equation*}
  F(s):=T^-_s\varphi(\gamma(s))-T^-_s\psi(\gamma(s)),\quad \forall s\in [0,t].
\end{equation*}
By assumption we have $F(t)>0$. If there is $\sigma\in[0,t)$ such that $F(\sigma)=0$ and $F(s)>0$ for all $s\in (\sigma,t]$, by definition we have
\begin{equation*}
  T^-_s\varphi(\gamma(s))\leq T^-_t\varphi(\gamma(\sigma))+\int_\sigma^sL(\gamma(\tau),T^-_\tau\varphi(\gamma(\tau)),\dot\gamma(\tau))d\tau,
\end{equation*}
and
\begin{equation*}
  T^-_s\psi(\gamma(s))=T^-_t\psi(\gamma(\sigma))+\int_\sigma^sL(\gamma(\tau),T^-_\tau\psi(\gamma(\tau)),\dot\gamma(\tau))d\tau,
\end{equation*}
which implies
\begin{equation*}
  F(s)\leq F(\sigma)+\lambda\int_\sigma^sF(\tau)d\tau,
\end{equation*}
where $F(\sigma)=0$. By the Gronwall inequality we conclude $F(s)\equiv 0$ for all $s\in[\sigma,t]$, which contradicts $F(t)>0$.

Therefore, for all $\sigma\in [0,t]$, we have $F(\sigma)>0$. Here $0<F(0)\leq \|\varphi-\psi\|_\infty$. By definition we have
\begin{equation*}
  T^-_s\varphi(\gamma(\sigma))\leq T^-_t\varphi(\gamma(0))+\int_0^\sigma L(\gamma(\tau),T^-_\tau\varphi(\gamma(\tau)),\dot\gamma(\tau))d\tau,
\end{equation*}
and
\begin{equation*}
  T^-_s\psi(\gamma(\sigma))=T^-_t\psi(\gamma(0))+\int_0^\sigma L(\gamma(\tau),T^-_\tau\psi(\gamma(\tau)),\dot\gamma(\tau))d\tau,
\end{equation*}
which implies
\begin{equation*}
  F(\sigma)\leq F(0)+\lambda\int_0^\sigma F(\tau)d\tau.
\end{equation*}
By the Gronwall inequality we get $F(\sigma)\leq \|\varphi-\psi\|_\infty e^{\lambda\sigma}$, which implies $T^-_t\varphi(x)-T^-_t\psi(x)\leq \|\varphi-\psi\|_\infty e^{\lambda t}$ by taking $\sigma=t$. Exchanging the role of $\varphi$ and $\psi$, we finally obtain that $\|T^-_t\varphi(x)-T^-_t\psi(x)\|_\infty\leq \|\varphi-\psi\|_\infty e^{\lambda t}$.

By definition, one can show the corresponding properties of $T^+$.
\end{proof}

Generally speaking, the local boundedness of $L(x,u,\dot x)$ does not hold if $H(x,u,p)$ satisfies the assumption (CER) rather than superlinearity. Fortunately, similar to \cite[Lemma 2.3]{Ish2}, one can prove the local boundedness of $L(x,u,\dot x)$ restricting on certain regions.
\begin{lemma}\label{CL}
Let $H(x,0,p)$ satisfy (C)(CON)(CER), there exist constants $\delta>0$ and $\bar{C}>0$ such that the Lagrangian $L(x,0,\dot x)$ associated to $H(x,0,p)$ satisfies
\begin{equation*}
L(x,0,\xi)\leq \bar{C},\quad \forall (x,\xi)\in M\times \bar B(0,\delta).
\end{equation*}
\end{lemma}

In the following part of this paper, we define
\begin{equation}\label{muu}
\mu:=\frac{\textrm{diam}(M)}{\delta}.
\end{equation}

\begin{lemma}\label{S2}
Let $\varphi\in C(M)$.
\begin{itemize}
\item [(1)] Given any $x_0\in M$, if $T^-_t\varphi(x_0)$ does not have an upper bound as $t\to +\infty$, then for any $c\in\mathbb R$, there exists $t_c>0$ such that $T^-_{t_c}\varphi(x)>\varphi(x)+c$ for all $x\in M$.

\item [(2)] Given any $x_0\in M$, if $T^-_t\varphi(x_0)$ does not have a lower bound as $t\to +\infty$, then for any $c\in\mathbb R$, there exists $t_c>0$ such that $T^-_{t_c}\varphi(x)<\varphi(x)+c$ for all $x\in M$.
\end{itemize}
\end{lemma}
\begin{proof} We only prove Item (1). Item (2) is similar to be verified.
 Assume that there exists $c_0\in\mathbb R$ such that for any $t>0$, we have a point $x_t\in M$ satisfying $T^-_t\varphi(x_t)\leq \varphi(x_t)+c_0$. Let $\alpha:[0,\mu]\rightarrow M$ be a geodesic connecting $x_t$ and $x$ with constant speed, where the constant $\mu$ was defined in (\ref{muu}), then $\|\dot \alpha\|\leq \delta$. If $T^-_{t+\mu}\varphi(x)>\varphi(x_t)+c_0$, since $T^-_t\varphi(x_t)\leq \varphi(x_t)+c_0$, there exists $\sigma\in[0,\mu)$ such that $T^-_{t+\sigma}\varphi(\alpha(\sigma))=\varphi(x_t)+c_0$ and  $T^-_{t+s}\varphi(\alpha(s))>\varphi(x_t)+c_0$ for all $s\in (\sigma,\mu]$. By definition we have
\begin{equation*}
\begin{aligned}
  T^-_{t+s}\varphi(\alpha(s))&\leq T^-_{t+\sigma}\varphi(\alpha(\sigma))+\int_\sigma^s L(\alpha(\tau),T^-_{t+\tau}\varphi(\alpha(\tau)),\dot \alpha(\tau))d\tau
  \\ &=\varphi(x_t)+c_0+\int_\sigma^s L(\alpha(\tau),T^-_{t+\tau}\varphi(\alpha(\tau)),\dot \alpha(\tau))d\tau,
\end{aligned}
\end{equation*}
which implies
\begin{equation*}
\begin{aligned}
  &T^-_{t+s}\varphi(\alpha(s))-(\varphi(x_t)+c_0)
  \leq \int_\sigma^s L(\alpha(\tau),T^-_{t+\tau}\varphi(\alpha(\tau)),\dot \alpha(\tau))d\tau
  \\ &\leq \int_\sigma^s L(\alpha(\tau),\varphi(x_t)+c_0,\dot \alpha(\tau))d\tau
  +\lambda\int_\sigma^s(T^-_{t+\tau}\varphi(\alpha(\tau))-(\varphi(x_t)+c_0))d\tau
  \\ &\leq L_0\mu+\lambda\int_\sigma^s(T^-_{t+\tau}\varphi(\alpha(\tau))-(\varphi(x_t)+c_0))d\tau,
\end{aligned}
\end{equation*}
where
\begin{equation*}
  L_0:=\bar{C}+\lambda \|\varphi+c_0\|_\infty,
\end{equation*}
and $\bar{C}$ is given in Lemma \ref{CL}.
By the Gronwall inequality, we have
\begin{equation*}
  T^-_{t+s}\varphi(\alpha(s))-(\varphi(x_t)+c_0)\leq L_0\mu e^{\lambda(s-\sigma)}\leq L_0\mu e^{\lambda \mu},\quad \forall s\in(\sigma,\mu].
\end{equation*}
Take $s=\mu$. We have $T^-_{t+\mu}\varphi(x)\leq \varphi(x_t)+c_0+L_0\mu e^{\lambda \mu}$. It means that $T^-_{t+\mu}\varphi(x)$ has an upper bound independent of $t$, which contradicts the assumption.
\end{proof}

\begin{lemma}\label{S3}
If there exist two continuous functions $\varphi_1$ and $\varphi_2$ on $M$ such that \[T^-_t\varphi_1\geq C_1,\quad T^-_t\varphi_2\leq C_2,\] then there is a constant function $\bar{\varphi}$ such that $|T^-_t\bar{\varphi}|\leq C_3$ for all $(x,t)\in M\times [0,+\infty)$, where $C_i$, $i=1,2,3$, are constants independent of $x$ and $t$.
\end{lemma}
\begin{proof}

Define $A_1:=\|\varphi_1\|_\infty$ and $A_2:=-\|\varphi_2\|_\infty$, then $A_2\leq A_1$ and \[T^-_tA_1(x)\geq T^-_t\varphi_1(x),\quad T^-_tA_2(x)\leq T^-_t\varphi_2(x)\quad \text{for\  all\ } x\in M.\] If $T^-_t A_1(x)$ has an upper bound independent of $t$, then $\bar{\varphi}\equiv A_1$ is enough. If $T^-_tA_1(x)$ does not have an upper bound independent of $t$, we define
\begin{equation*}
  A^*:=\inf\{A:\ \exists t_A>0\ \textrm{such}\ \textrm{that}\ T^-_{t_A}A(x)\geq A,\ \forall x\in M\}.
\end{equation*}
By using Lemma \ref{S2} (1) with $c=0$, we have $A^*\leq A_1<+\infty$. The remaining discussion is divided into two cases.

\noindent Case (1): $A^*>-\infty$. In this case,  we aim to prove that $\bar{\varphi}\equiv A^*$ is enough.

We first show that $T^-_tA^*(x)$ has an upper bound independent of $t$.  Assume $T^-_tA^*(x)$ does not have an upper bound. By Lemma \ref{S2} (1), for $c=1$, there is $t_1>0$ such that $T^-_{t_1} A^*(x)>A^*+1$ for all $x\in M$. By Proposition \ref{psg} (2), for any $\varepsilon >0$, we have \[T^-_{t_1}(A^*-\varepsilon)(x)\geq T^-_{t_1}A^*(x)-e^{\lambda t_1}\varepsilon>A^*+1-e^{\lambda t_1}\varepsilon.\]
 For every $0<\varepsilon<(e^{\lambda t_1}-1)^{-1}$, we have $T^-_{t_1}(A^*-\varepsilon)(x)>A^*-\varepsilon$. It means that we have found a smaller constant $A^*-\varepsilon$ such that if  $t_{A^*-\varepsilon}:=t_1$, then \[T^-_{t_{A^*-\varepsilon}}(A^*-\varepsilon)(x)>A^*-\varepsilon,\] which contradicts the definition of $A^*$.

We then prove that $T^-_tA^*$ has a lower bound independent of $t$.  Assume $T^-_tA^*(x)$ does not have a lower bound. By using Lemma \ref{S2} (2) with $c=-1$, there is $t_1>0$ such that $T^-_{t_1} A^*(x)<A^*-1$ for all $x\in M$. Since $T^-_tA^*(x)$ has an upper bound independent of $t$, then $A^*<A_1$. By Proposition \ref{psg} (2) and $A^*<A_1$, there is a constant $\delta_0>0$ such that $A^*+\delta<A_1$ and
\begin{equation}\label{s1}
  T^-_{t_1}(A^*+\delta)(x)<A^*-\frac{1}{2}+\delta<A^*+\delta,
\end{equation}
for all $\delta\in[0,\delta_0)$. By the definition of $A^*$, there is $\bar A\in [A^*,A^*+\delta_0)$ and $t_2:=t_{\bar A}>0$ such that
\begin{equation}\label{s2}
  T^-_{t_2}\bar A(x)\geq \bar A.
\end{equation}
By (\ref{s1}), we have
\begin{equation}\label{bara}
 T^-_{t_1}\bar{A}(x)<\bar{A}-\frac{1}{2}<\bar{A}.
\end{equation}
Define $B^*:=\bar A-\frac{1}{2}$. According to the continuity of $T^-_t\varphi(x)$ at $t=0$, there exists $\varepsilon_0>0$ such that for $0\leq\sigma<\varepsilon_0$, we have
\begin{equation}\label{s3}
  T^-_\sigma B^*(x)\leq \bar A-\frac{1}{4}.
\end{equation}
For $t_1$ and $t_2>0$, there exist $n_1$ and $n_2\in\mathbb N$, and $\varepsilon\in[0,\varepsilon_0)$ such that $n_1t_1+\varepsilon=n_2t_2$. By Proposition \ref{psg} (1) and (\ref{s1}), we have
\begin{equation}\label{s4}
  T^-_{n_1t_1}\bar A(x)\leq T^-_{t_1}\bar A(x)<B^*.
\end{equation}
Take $\sigma=\varepsilon$ in (\ref{s3}). By Proposition \ref{psg} (1) and  (\ref{s4}), we get
\begin{equation}\label{s6}
  T^-_\varepsilon\circ T^-_{n_1t_1}\bar A(x)\leq T^-_\varepsilon B^*(x)\leq \bar A-\frac{1}{4}.
\end{equation}
By (\ref{s2}), one has $T^-_{n_2t_2}\bar A(x)\geq \bar A$. Thus
\begin{equation}\label{s7}
  \bar{A}-\frac{1}{4}\geq T^-_\varepsilon\circ T^-_{n_1t_1}\bar A(x)=T^-_{n_2t_2}\bar A(x)\geq \bar A,
\end{equation}
which is a contradiction.


\noindent Case (2): $A^*=-\infty$. In this case, we aim to prove that for any $A<A_2$, the function $T^-_tA(x)$ is uniformly bounded. Namely, $\bar{\varphi}\equiv A$ is enough. Since $T^-_tA(x)\leq T^-_tA_2(x)$, then $T^-_tA(x)$ has an upper bound. The proof of the existence of the lower bound of $T^-_tA(x)$ is similar to Case (1). In fact, we only need to replace $A^*$, $A_1$  by $A$ and  $A_2$ respectively.
\end{proof}

\begin{remark}\label{S1}
Let $\varphi\in C(M)$. According to \cite[Theorem 6.1]{Ish6},
 if $T^-_t\varphi(x)$ has a bound independent of $t$, then  the lower half limit
\[\check{\varphi}(x):=\lim_{r\rightarrow 0+}\inf\{T^-_t\varphi(y):\ d(x,y)<r,\ t>1/r\}\]
is a Lipschitz  solution of (\ref{E}). According to Proposition \ref{fix}, the function $\check{\varphi}$ is a backward weak KAM solution of (\ref{E}). Similarly, if $T^+_t\varphi(x)$ has a bound independent of $t$, define
\begin{equation*}
  \begin{aligned}
  \hat{\varphi}(x):&=\lim_{r\rightarrow 0+}\sup\{T^+_t\varphi(y):\ d(x,y)<r,\ t>1/r\}
  \\ &=\lim_{r\rightarrow 0+}\sup\{-\bar T^-_t(-\varphi)(y):\ d(x,y)<r,\ t>1/r\}
  \\ &=-\lim_{r\rightarrow 0+}\inf\{\bar T^-_t(-\varphi)(y):\ d(x,y)<r,\ t>1/r\}.
  \end{aligned}
\end{equation*}
Then $-\hat \varphi$ is a Lipschitz  solution of $H(x,-u,-\partial_xu)=0$. Equivalently, $\hat \varphi$ is a forward weak KAM solution of (\ref{E}).
\end{remark}

\noindent{\it Proof of Theorem  \ref{S}.}
 By assumption, there is $\varphi\in C(M)$ and $t_1>0$ such that $T^-_{t_1}\varphi\geq \varphi$. For any $t>0$, one can find $n\in\mathbb N$ and $r\in[0,t_1)$ such that $t=nt_1+r$. By Proposition \ref{psg} (1), we have $T^-_t\varphi\geq T^-_r\varphi$. Namely, $T^-_t\varphi$ has a lower bound independent of $t$. On the other hand, there is $\psi\in C(M)$ and $t_2>0$ such that $T^-_{t_2}\psi\leq \psi$. It is similar to obtain that $T^-_t\psi$ has an upper bound independent of $t$. By Lemma \ref{S3}, there exists a constant function $\bar{\varphi}$ such that $T^-_t\bar{\varphi}$ is uniformly bounded. By Remark \ref{S1}, (\ref{E}) admits Lipschitz viscosity solutions.
 \qed

\section{The Aubry set}\label{pm2'}

Let $u_-\in\mathcal S_-$. At the beginning, we prove that the limit function $x\mapsto \lim_{t\to +\infty}T^+_tu_-(x)$ is well defined. Corollaries \ref{<u-} and \ref{up} guarantee the boundedness of $T^+_tu_-$. Moreover, Item (1) of Theorem \ref{m2'}  is verified by Proposition \ref{T+cu+}, and Item (3) is shown by Proposition \ref{xt}. Item (2) of Theorem \ref{m2'} is similar to Item (1).

\begin{proposition}\label{*'}
Let $\varphi\in C(M)$ and $u_-\in \mathcal{S}_-$. If $\varphi$ satisfies the following condition:
\begin{itemize}
\item [($\odot$)] $\varphi\leq u_-$ and there exists a point $x_0$ such that $\varphi(x_0)=u_-(x_0)$.
\end{itemize}
\noindent then $T^+_t\varphi(x)$ has a bound independent of $t$ and $\varphi$.
\end{proposition}
We divide the proof into three parts, that is, Lemmas \ref{51}, \ref{52} and \ref{53}.
\begin{lemma}\label{51}
Suppose $\varphi$ satisfies the condition ($\odot$), then $T^+_t\varphi(x)\leq u_-(x)$ for all $t>0$.
\end{lemma}
\begin{proof}
 Assume there exists $(x,t)\in M\times (0,+\infty)$ such that $T^+_t\varphi(x)>u_-(x)$. Let $\gamma:[0,t]\rightarrow M$  be a minimizer of $T^+_t\varphi(x)$ with $\gamma(0)=x$. Define
\begin{equation*}
  F(s)=T^+_{t-s}\varphi(\gamma(s))-u_-(\gamma(s)),\quad s\in [0,t].
\end{equation*}
Then $F(s)$ is continuous and $F(t)=\varphi(\gamma(t))-u_-(\gamma(t))\leq 0$. By assumption we have $F(0)>0$. Then there is $\tau_0\in(0,t]$ such that $F(\tau_0)=0$ and $F(\tau)>0$ for all $s\in [0,\tau_0)$. For each $\tau\in[0,\tau_0]$, we have
\begin{equation*}
  T^+_{t-\tau}\varphi(\gamma(\tau))=T^+_{t-\tau_0}\varphi(\gamma(\tau_0))-\int_{\tau}^{\tau_0}L(\gamma(s),T^+_{t-s}\varphi(\gamma(s)),\dot{\gamma}(s))ds.
\end{equation*}
Since $u_-=T^-_tu_-$ for all $t>0$, we have
\begin{equation*}
  u_-(\gamma(\tau_0))\leq u_-(\gamma(\tau))+\int_{\tau}^{\tau_0}L(\gamma(s),u_-(\gamma(s)),\dot{\gamma}(s))ds.
\end{equation*}
Thus $F(\tau)\leq F(\tau_0)+\lambda \int_{\tau}^{\tau_0} F(s)ds$, where $F(\tau_0)=0$. Define $F(s)=G(\tau_0-s)$. We get
\begin{equation*}
  G(\tau_0-\tau)\leq \lambda\int_0^{\tau_0-\tau}G(\sigma)d\sigma.
\end{equation*}
By the Gronwall inequality, we conclude $F(\tau)=G(\tau_0-\tau)\equiv 0$ for all $\tau\in [0,\tau_0]$, which contradicts $F(0)>0$.
\end{proof}
\begin{corollary}\label{<u-}
Let $u_-\in\mathcal S_-$. Then $T^+_tu_-\leq u_-$ for each $t>0$.
\end{corollary}

Combining Corollary \ref{<u-} with Proposition \ref{psg} (1), one can  obtain that $T^+_tu_-=T^+_s\circ T^+_{t-s}u_-\leq T^+_su_-$ for all $t>s$, then we have
\begin{corollary}\label{ti}
$T^+_tu_-$ is decreasing in $t$.
\end{corollary}

\begin{lemma}\label{52}
Suppose $\varphi$ satisfies the condition ($\odot$). Let $\gamma_-:(-\infty,0]\rightarrow M$ be a $(u_-,L,0)$-calibrated curve  with $\gamma_-(0)=x_0$, then $T^+_t\varphi(\gamma_-(-t))=u_-(\gamma_-(-t))$ for each $t>0$.
\end{lemma}
\begin{proof}
For each $t>0$, we define $\gamma_t(s):=\gamma_-(s-t)$ for $s\in [0,t]$. By Lemma \ref{51}, for each $s\in[0,t]$, we have $u_-(\gamma_t(s))\geq T^+_{t-s}\varphi(\gamma_t(s))$. Define
\begin{equation*}
  F(s)=u_-(\gamma_t(s))-T^+_{t-s}\varphi(\gamma_t(s)),
\end{equation*}
then $F(s)\geq 0$ and $F(t)=0$. If $F(0)>0$, then there is $s_0\in(0,t]$ such that $F(s_0)=0$ and $F(s)>0$ for all $s\in [0,s_0)$. By definition, for $s_1\in[0,s_0)$, we have
\begin{equation*}
  u_-(\gamma_t(s_0))-u_-(\gamma_t(s_1))=\int_{s_1}^{s_0}L(\gamma_t(s),u_-(\gamma_t(s)),\dot \gamma_t(s))ds,
\end{equation*}
and
\begin{equation*}
  T^+_{t-s_1}\varphi(\gamma_t(s_1))\geq T^+_{t-s_0}\varphi(\gamma_t(s_0))-\int_{s_1}^{s_0}L(\gamma_t(s),T^+_{t-s}\varphi(\gamma_t(s)),\dot \gamma_t(s))ds,
\end{equation*}
which implies
\begin{equation*}
  F(s_1)\leq F(s_0)+\lambda \int_{s_1}^{s_0} F(s)ds.
\end{equation*}
By the Gronwall inequality, we conclude $F(s)\equiv 0$ for all $s\in[0,s_0]$, which contracts $F(0)>0$. Therefore $F(0)=0$. Namely, $T^+_t\varphi(\gamma_t(0))=u_-(\gamma_t(0))$. Recall $\gamma_t(s):=\gamma_-(s-t)$. We have $T^+_t\varphi(\gamma_-(-t))=u_-(\gamma_-(-t))$.
\end{proof}

\begin{lemma}\label{53}
Suppose $\varphi$ satisfies the condition ($\odot$), then $T^+_t\varphi(x)$ has a lower bound independent of $t$ and $\varphi$.
\end{lemma}
\begin{proof}
Let $\gamma_-:(-\infty,0]\rightarrow M$  be a $(u_-,L,0)$-calibrated curve with $\gamma_-(0)=x_0$. Let $t>\mu$ and $\alpha:[0,\mu]\rightarrow M$ be a geodesic connecting $x$ and $\gamma_-(-t+\mu)$ with constant speed, then $\|\dot \alpha\|\leq \delta$. If $T^+_t\varphi(x)\geq u_-(\gamma_-(-t+\mu))$, then the proof is completed. It remains to consider $T^+_t\varphi(x)<u_-(\gamma_-(-t+\mu))$. Since \[T^+_{t-\mu}\varphi(\gamma_-(-t+\mu))=u_-(\gamma_-(-t+\mu)),\] then there is $\sigma\in(0,\mu]$ such that \[T^+_{t-\sigma}\varphi(\alpha(\sigma))=u_-(\gamma_-(-t+\mu)), \quad T^+_{t-s}\varphi(\alpha(s))<u_-(\gamma_-(-t+\mu))\quad \text{for\  all\ } s\in [0,\sigma).\] By definition we have
\begin{equation*}
\begin{aligned}
  T^+_{t-s}\varphi(\alpha(s))&\geq T^+_{t-\sigma}\varphi(\alpha(\sigma))-\int_s^\sigma L(\alpha(\tau),T^+_{t-\tau}\varphi(\alpha(\tau)),\dot \alpha(\tau))d\tau
  \\ &=u_-(\gamma_-(-t+\mu))-\int_s^\sigma L(\alpha(\tau),T^+_{t-\tau}\varphi(\alpha(\tau)),\dot \alpha(\tau))d\tau,
\end{aligned}
\end{equation*}
which implies
\begin{equation*}
\begin{aligned}
  &u_-(\gamma_-(-t+\mu))-T^+_{t-s}\varphi(\alpha(s))
  \leq \int_s^\sigma L(\alpha(\tau),T^+_{t-\tau}\varphi(\alpha(\tau)),\dot \alpha(\tau))d\tau
  \\ &\leq \int_s^\sigma L(\alpha(\tau),u_-(\gamma_-(-t+\mu)),\dot \alpha(\tau))d\tau
  +\lambda\int_s^\sigma(u_-(\gamma_-(-t+\mu))-T^+_{t-\tau}\varphi(\alpha(\tau)))d\tau
  \\ &\leq L_0\mu+\lambda\int_s^\sigma(u_-(\gamma_-(-t+\mu))-T^+_{t-\tau}\varphi(\alpha(\tau)))d\tau,
\end{aligned}
\end{equation*}
where
\begin{equation*}
  L_0:=\bar{C}+\lambda \|u_-\|_\infty,
\end{equation*}
and $\bar{C}$ is given by Lemma \ref{CL}.
Let $G(\sigma-s)=u_-(\gamma_-(-t+\mu))-T^+_{t-s}\varphi(\alpha(s))$, then
\begin{equation*}
  G(\sigma-s)\leq L_0\mu+\lambda\int_0^{\sigma-s}G(\tau)d\tau.
\end{equation*}
By the Gronwall inequality, we have
\begin{equation*}
  u_-(\gamma_-(-t+\mu))-T^+_{t-s}\varphi(\alpha(s))=G(\sigma-s)\leq L_0\mu e^{\lambda(\sigma-s)}\leq L_0\mu e^{\lambda \mu},\quad \forall s\in[0,\sigma).
\end{equation*}
Thus $T^+_t\varphi(x)\geq u_-(\gamma_-(-t+\mu))-L_0\mu e^{\lambda \mu}$. We finally get a lower bound of $T^+_t\varphi(x)$ independent of $t$ and $\varphi$.
\end{proof}
\begin{corollary}\label{up}
$T^+_tu_-$ has a lower bound independent of $t$.
\end{corollary}

\begin{proposition}\label{T+cu+}
$T^+_t u_-$ converges to a forward weak KAM solution $u_+$ of (\ref{E}) uniformly as $t\to +\infty$.
\end{proposition}
\begin{proof}
We first recall that $T^+_t\varphi:=-\bar T^-_t(-\varphi)$, where $\bar T^-_t$ denotes the backward Lax-Oleinik semigroup associated to $L(x,-u,-\dot{x})$. Since $T^+_tu_-$ is decreasing in $t$, the function $u(x,t):=\bar T^-_t(-u_-)$ is increasing in $t$. Thus, $\partial_t u(x,t)\geq 0$ holds in the viscosity sense. Since $u(x,t)$ is the viscosity solution of $\partial_t u+H(x,-u,-\partial_x u)=0$, we have $H(x,-u,-\partial_x u)\leq 0$. Since $T^+_tu_-$ has a bound independent of $t$, $u(x,t)$ has a bound independent of $t$. We conclude that $\|\partial_x T^+_tu_-\|_\infty=\|\partial_x u(x,t)\|_\infty$ has a bound independent of $t$ by (CER). Corollaries \ref{ti} and \ref{up} imply that the pointwise limit $u_+(x)=\lim_{t\rightarrow+\infty}T^+_tu_-(x)$ exists. Since $\|\partial_x T^+_tu_-\|_\infty$ has a bound independent of $t$, the limit function $u_+$ is continuous. By the Dini theorem, the family $T^+_t u_-$ converges uniformly  to $u_+$. It remains to prove that $u_+$ is a fixed point of $T^+_t$. For each $t>0$, by Proposition \ref{psg} (2), we have
\[\|T^+_{t+s}u_--T^+_tu_+\|_\infty\leq e^{\lambda t}\|T^+_s u_--u_+\|_\infty.\]
Letting $s\to+\infty$, we get $T^+_t u_+=u_+$.
\end{proof}

\begin{proposition}\label{xt}
The set $\mathcal I_{u_-}$ is nonempty. More precisely, let $\gamma_-:(-\infty,0]\rightarrow M$ be a $(u_-,L,0)$-calibrated curve. Define
\[\alpha(\gamma_-):=\{x\in M:\ \textrm{there\ exists\ a\ sequence}\ t_n\rightarrow-\infty\ \textrm{such\ that}\ \text{d}(\gamma_-(t_n),x)\to 0\}.\]
Then $\alpha(\gamma_-)$ is nonempty, and it is contained in $\mathcal I_{u_-}$.
\end{proposition}
\begin{proof}
Let $\gamma_-:(-\infty,0]\rightarrow M$ be a $(u_-,L,0)$-calibrated curve. By Lemma \ref{52}, for each $t>0$ we have \[T^+_tu_-(\gamma_-(-t))=u_-(\gamma_-(-t)).\]  Since $M$ is compact, the set $\alpha(\gamma_-)$ is nonempty. Let $x^*\in \alpha(\gamma_-)$ and $t_n\rightarrow+\infty$ such that $d(\gamma_-(-t_n),x^*)\to 0$. The following inequality holds
\begin{align*}
  |T^+_{t_n}u_-(\gamma_-(-t_n))-u_+(x^*)|\leq& |T^+_{t_n}u_-(\gamma_-(-t_n))-u_+(\gamma_-(-t_n))|\\
  &+|u_+(\gamma_-(-t_n))-u_+(x^*)|.
\end{align*}
 The function $u_+$ is Lipschitz (see Proposition \ref{u<LLip}). Thus, as $t_n\to +\infty$,
 \[|u_+(\gamma_-(-t_n))-u_+(x^*)|\to 0.\]
Since $T^+_t u_-$ converges to $u_+$ uniformly, then
\[|T^+_{t_n}u_-(\gamma_-(-t_n))-u_+(\gamma_-(-t_n))|\to 0.\]
 Therefore, the limit of $T^+_{t_n}u_-(\gamma_-(-t_n))$ is $u_+(x^*)$. On the other hand, we have \[T^+_{t_n}u_-(\gamma_-(-t_n))=u_-(\gamma_-(-t_n)),\] which tends to $u_-(x^*)$ by the continuity of $u_-$. We conclude that $u_+(x^*)=u_-(x^*)$. It means $\alpha(\gamma_-)\subseteq \mathcal I_{u_-}$.
\end{proof}

\section{A  comparison result for the solutions of (\ref{E})}\label{pm27}

According to \cite[Theorem 3.2]{inc}, the viscosity solution of
\[H(x,-u(x),-\partial_x u(x))=0\]
is unique. By Proposition \ref{fix}, the forward weak KAM solution $u_+$ of (\ref{E}) is also unique. Define $u_-=\lim_{t\rightarrow+\infty}T^-_t u_+$, then the conjugate pair $(u_-,u_+)$ is unique. According to Proposition \ref{T+cu+}, $T^+_t v_-$ converges to the unique forward weak KAM solution $u_+$ uniformly as $t\rightarrow+\infty$ and $u_+\leq v_-$ for all $v_-\in \mathcal{S}_-$.

\bigskip

\noindent\textit{Proof of Theorem \ref{four3}.} We first prove the result (1). By Proposition \ref{xt}, the set $\mathcal I_{v_-}$ is nonempty for each $v_-\in\mathcal S_-$. For $x\in\mathcal I_{v_2}$, we have \[u_+(x)\leq v_1(x)\leq v_2(x)=u_+(x),\] then $v_1(x)=v_2(x)=u_+(x)$, that is, $x\in\mathcal I_{v_1}$.

We then prove the result (2). For each $x\in M$, let $\gamma_2:(-\infty,0]\rightarrow M$ be a $(v_2,L,0)$-calibrated curve  with $\gamma_2(0)=x$. By Proposition \ref{xt}, there is a $t_0>0$ large enough, such that $\gamma_2(-t_0)\in\mathcal O$, where  $\mathcal O$ denotes a neighborhood of $\mathcal I_{v_2}$. Define
\begin{equation*}
  F(s)=v_1(\gamma_2(s))-v_2(\gamma_2(s)),\quad s\in[-t_0,0].
\end{equation*}
 If $v_1(x)>v_2(x)$, then $F(0)=v_1(x)-v_2(x)>0$ and $F(-t_0)=v_1(\gamma_2(-t_0))-v_2(\gamma_2(-t_0))\leq 0$. Then there is $\sigma\in [-t_0,0)$ such that $F(\sigma)=0$ and $F(s)>0$ for all $s\in(\sigma,0]$. By definition we have
\begin{equation*}
  v_1(\gamma_2(s))-v_1(\gamma_2(\sigma))\leq \int_\sigma^s L(\gamma_2(\tau),v_1(\gamma_2(\tau)),\dot \gamma_2(\tau))d\tau,
\end{equation*}
and
\begin{equation*}
  v_2(\gamma_2(s))-v_2(\gamma_2(\sigma))=\int_\sigma^s L(\gamma_2(\tau),v_2(\gamma_2(\tau)),\dot \gamma_2(\tau))d\tau,
\end{equation*}
which implies
\begin{equation*}
  F(s)\leq F(\sigma)+\lambda \int_\sigma^s F(\tau)d\tau.
\end{equation*}
By the Gronwall inequality we conclude $F(s)\equiv 0$ for all $s\in[\sigma,0]$, which contradicts $F(0)>0$. We  conclude $v_1\leq v_2$ on $M$.

The result (3) follows directly from (2). The proof is now complete.
\qed


\section{On the example (\ref{E0})}\label{exxxam}

Let $u_+$ be the unique forward weak KAM solution of (\ref{E0}). We have already known that $u_+\leq v_-$ for each viscosity solution $v_-$ of (\ref{E0}). It is sufficient to show $u_+(x)<u_2(x)$ for all $x\in(-1,1]\backslash \{0\}$. By the symmetry of $u_2$, we only need to consider $x\in (0,1]$.

By \cite[Theorem 5.3.6]{cann} and Proposition \ref{fix}, each $u_+$ is a semiconvex function with linear modulus. Note that $u_+(x)\leq u_2(x)$. Moreover, $u_+$ can not be equal to $u_2$ at $x=1$. In fact, if $u_2=u_+$ at $x=1$, combining with the semiconcavity of $u_2$, then $u_2$ is differentiable at this point. Let us recall
\[u_2(x)=\frac{\lambda-\sqrt{\lambda^2-4}}{2}V(x),\]
 and $V$ is not differentiable at $x=1$. This is a contradiction.

We then assume that there exists $x_0\in (0,1)$ such that $u_+(x_0)=u_2(x_0)$. Since $u_2(x)$ is differentiable for each $x\in [0,1)$, then $u_2$ satisfies
\[-\lambda u(x)+\frac{1}{2}|u'(x)|^2+V(x)=0\]
in the classical sense for $x\in [0,1)$. Note that $|u'_2(x)|>0$ for $x\in (0,1)$, we have $\lambda u_2(x)>V(x)$ for all $x\in (0,1)$. For $z>V(x)$, we set
\begin{equation*}
  f(x,z):=\lambda \sqrt{2(z-V(x))},
\end{equation*}
then the function $(x,z)\mapsto f(x,z)$ is of class $C^1$ on
\[\{(x,z)\in \R^2\ |\ x\in (0,1),\ z>V(x)\}.\]
Given $\varepsilon\in (0,x_0)$, denote
\[\Omega_\varepsilon:=\left\{(x,z)\in \R^2\ |\ x\in [\eps,1),\ z\in \left[\frac{1}{2}\lambda u_2(x)+\frac{1}{2}V(x),\frac{3}{2}\lambda u_2(x)-\frac{1}{2}V(x)\right]\right\}.\]
It follows that
\[\left|\frac{\partial f}{\partial z}\right|=\frac{\lambda}{\sqrt{2(z-V(x))}}\leq \frac{\lambda}{\sqrt{\lambda u_2(\eps)-V(\eps)}}<+\infty.\]
 By the classical theory of ordinary differential equations, for $x_0\in (0,1)$, $\lambda u_2(x)$ is the unique solution of
\begin{equation}\label{ode}
  \frac{dz}{dx}=f(x,z),\quad z(x_0)=\lambda u_2(x_0),\quad \text{on}\ \Omega_\eps.
\end{equation}

We assert that $u_+$ is differentiable on $(0,1)$. If the assertion is true, then $u_+$ satisfies (\ref{E0}) in the classical sense. Since $u_+\leq u_2$ and $u_+(x_0)=u_2(x_0)$, $\lambda u_+$ is the unique solution of (\ref{ode}) on $\Omega_\eps$. That is, $u_+=u_2$ on $(\eps,1)$. Moreover, $u_+=u_2$ on $\mathbb{S}$ by continuity and the arbitrariness of $\eps$. This contradicts the semiconvexity of $u_+$. Therefore, we have $u_+(x)<u_2(x)$ for all $x\in (0,1]$.

It remains to show that $u_+$ is differentiable on $(0,1)$. Assume there exists $y_0\in (0,1)$ such that $u_+$ is not differentiable at $y_0$. By \cite[Lemma 2.2]{WWY4}, \cite[Theorem 3.3.6]{cann}, and Proposition \ref{fix}, we have
\[D^*u_+(x)=\{p\in D^-u_+(x)\ |\ H(x,u_+(x),p)=0\}, \quad D^-u_+(x)=\text{co}D^*u_+(x),\]
where $D^*$ stands for the set of all reachable gradients and ``co" denotes the convex hull.
It follows from (\ref{E0}) that
\[D^*u_+(y_0)=\{\pm l\},\]
where $l$ is a positive constant.
  By the semiconvexity of $u_+$, there exists $y_1\in (0,y_0)$ such that $u_+(y_1)>u_+(y_0)$. Moreover, there is $z_0\in (0,y_0)$  achieving a local maximum of $u_+$. By using the semiconvexity of $u_+$ again, it is differentiable at $z_0$, then $u'_+(z_0)=0$. By (\ref{E0}), we have \[-\lambda u_+(z_0)+V(z_0)=0.\] Since $u'_+(x)$ exists for almost all $x$, there is $z_1\in (z_0,y_0)$ such that $u'_+(z_1)$ exists. By the Newton-Leibniz formula, one can require $|u'_+(z_1)|> 0$ and $u_+(z_0)\geq u_+(z_1)\geq 0$. By definition, we  have $V(z_1)>V(z_0)$. Therefore
\begin{equation*}
  -\lambda u_+(z_1)+\frac{1}{2}|u'_+(z_1)|^2+V(z_1)>-\lambda u_+(z_0)+V(z_0)=0,
\end{equation*}
which contradicts that $u_+$ satisfies (\ref{E0}) at $z_1$ in the classical sense.
\qed

\bigskip

{\bf Acknowledgement}\ \ The authors would like to thank Professor Wei Cheng and Kaizhi Wang for many useful discussions.

\appendix

\section{One dimensional variational problems}\label{preli}

The following results are useful in the proof of the existence and regularity of the minimizers in (\ref{T-}), which all come from \cite{One} and \cite{gam}. The results in \cite{One,gam} were  stated for the case in the Euclidean space $\mathbb R^n$. It is not difficult to generalize them for the case in the Riemannian manifold $M$.

\subsection{$\Gamma$-convergence}\label{A.11}
\begin{lemma}\textbf{}\label{TM}
Let $J$ be a bounded interval. Assume that $F(t,x,\dot x)$ is lower semicontinuous, convex in $\dot x$, and has a lower bound. Then the integral functional
\begin{equation*}
  \mathcal F(\gamma)=\int_J F(s,\gamma(s),\dot \gamma(s))ds
\end{equation*}
is sequentially weakly lower semicontinuous in $W^{1,1}(J,M)$.
\end{lemma}

\begin{proposition}\textbf{}\label{Tonelli}
Let $M$ be a compact connected smooth manifold. Denote by $I=(a,b)\subset R$ a bounded interval, and let $F(t,x,\dot{x})$ be a Lagrangian defined on $I\times TM$. Assume $F$ satisfies
\begin{itemize}
\item [(i)] $F(t,x,\dot x)$ is measurable in $t$ for all $(x,\dot x)$, and continuous in $(x,\dot x)$ for almost every $t$;

\item [(ii)] $F(t,x,\dot{x})$ is convex in $\dot{x}$;

\item [(iii)] $F(t,x,\dot{x})$ is superlinear in $\dot{x}$.
\end{itemize}
\noindent Then for any given boundary condition $x_0$ and $x_1\in M$, there exists a minimizer of $\int_I F(t,x,\dot{x})dt$ in $\{x(t)\in W^{1,1}([a,b],M):\ x(a)=x_0,\ x(b)=x_1\}$.
\end{proposition}
\begin{definition}\textbf{}
Let $X$ be a topological space. Given a sequence $F_n:X\rightarrow [-\infty,+\infty]$, then we define
\begin{equation*}
  (\Gamma-\liminf_{n\rightarrow +\infty} F_n)(x)=\sup_{U\in \mathcal N(x)}\liminf_{n\rightarrow +\infty}\inf_{y\in U} F_n(y),
\end{equation*}
\begin{equation*}
  (\Gamma-\limsup_{n\rightarrow +\infty} F_n)(x)=\sup_{U\in \mathcal N(x)}\limsup_{n\rightarrow +\infty}\inf_{y\in U} F_n(y).
\end{equation*}
Here the neighbourhoods $\mathcal N(x)$ can be replaced by the topological basis. When the superior limit equals to the inferior limit, we can define the $\Gamma$-limit.
\end{definition}
\begin{definition}
Let $X$ be a topological space. For every function $F:X\rightarrow [-\infty,+\infty]$, the lower semicontinuous envelope $sc^- F$ of $F$ is defined for every $x\in X$ by
\begin{equation*}
  (sc^-F)(x)=\sup_{G\in\mathcal G(F)} G(x),
\end{equation*}
where $\mathcal G(F)$ is the set of all lower semicontinuous functions $G$ on $X$ such that $G(y)\leq F(y)$ for every $y\in X$.
\end{definition}
\begin{lemma}\label{inc}
If $F_n$ is an increasing sequence, then
\begin{equation*}
  \Gamma-\lim_{n\rightarrow +\infty}F_n=\lim_{n\rightarrow +\infty} sc^- F_n=\sup_{n\in\mathbb N} sc^- F_n.
\end{equation*}
\end{lemma}
\begin{remark}\label{rega}
If $F_n$ is an increasing sequence of lower semicontinuous functions which converges pointwisely to a function $F$, then
$F$ is lower semicontinuous and $F_n$ has a $\Gamma$-convergence to $F$ by Lemma \ref{inc}.
\end{remark}
\begin{lemma}\label{inf}
If the sequence $F_n$ has a $\Gamma$-convergence in $X$ to $F$, and there is a compact set $K\subset X$ such that
\begin{equation*}
  \inf_{x\in X} F_n(x)=\inf_{x\in K}F_n(x),
\end{equation*}
then $F$ takes its minimum in $X$, and
\begin{equation*}
  \min_{x\in X}F(x)=\lim_{n\rightarrow +\infty}\inf_{x\in X} F_n(x).
\end{equation*}
\end{lemma}

\subsection{ Regularity of minimizers in $t$-dependent cases}\label{A.2}

The following results focus on the regularity of minimizers. Consider the following one dimensional variational problem
\begin{equation}\label{P}\tag{P}
  I(\gamma):=\int_a^b F(t,\gamma(t),\dot{\gamma}(t))dt+\Psi(\gamma(a),\gamma(b)),
\end{equation}
where $\gamma$ is taken in the class of absolutely continuous curves. $\Psi$ takes its value in $\mathbb R\cup \{+\infty\}$ and stands for the constraints on the two ends of the curves $\gamma$.

In the following, we focus on a certain minimizer of the above integral functional, which is denoted by $\gamma_*\in W^{1,1}([a,b],M)$. Due to the Lavrentiev phenomenon, the minimizier may not be Lipschitz. One can refer \cite{ball} for various counterexamples. Thanks to \cite{Bet},  the Lipschitz regularity of the minimizers still holds for $F:=L(x,v(x,t),\dot{x})$, where $v(x,t)$ is a Lipschitz function  (see Lemma \ref{3.1} (1)).  Let us recall the related results in  \cite{Bet} as follows.

\begin{itemize}
\item [\textbf{($\diamondsuit$):}] $F$ takes its value in $\mathbb R$, there exist a constant $\varepsilon>0$ and a Lebesgue-Borel-measurable map $k:[a,b]\times(0,+\infty)\rightarrow \mathbb R$ such that $k(t,1)\in L^1[a,b]$, and, for a.e. $t\in[a,b]$,  all $\sigma>0$
    \begin{equation*}
      |F(t_2,\gamma_*(t),\sigma\dot{\gamma_*}(t))-F(t_1,\gamma_*(t),\sigma\dot{\gamma_*}(t))|\leq k(t,\sigma)|t_2-t_1|,
    \end{equation*}
    where $t_{1},t_2\in[t-\varepsilon,t+\varepsilon]\cap [a,b]$.
\end{itemize}

\begin{lemma}\label{W4.1}
Let $\gamma_*$ be a minimizer of (\ref{P}). If $F$ satisfies ($\diamondsuit$), then there exists an absolutely continuous function $p\in W^{1,1}([a,b],\mathbb R)$ such that for a.e. $t\in[a,b]$, we have
\begin{equation}\label{W}\tag{W}
  F\left(t,\gamma_*(t),\frac{\dot{\gamma}_*(t)}{v}\right)v-F(t,\gamma_*(t),\dot{\gamma}_*(t))\geq p(t)(v-1),\quad \forall v>0,
\end{equation}
and $ |p'(t)|\leq k(t,1)$ for a.e. $ t\in[a,b]$.
\end{lemma}

\begin{lemma}\label{Lip6.3}
Let $\gamma_*$ be a minimizer of (\ref{P}). Assume $F$ is a Borel measurable function. If $F$ satisfies ($\diamondsuit$) and
\begin{itemize}
\item [(1)] Superlinearity: there exists a function $\Theta:[0,+\infty)\rightarrow\mathbb R$ satisfying
\begin{equation*}
  \lim_{r\rightarrow+\infty}\frac{\Theta(r)}{r}=+\infty,\quad \textrm{and}\quad F(t,\gamma_*(t),\xi)\geq \Theta(\|\xi\|)\quad \textrm{for\ all}\ \xi\in T_{\gamma_*(t)}M.
\end{equation*}

\item [(2)] Local boundedness: there exists $\rho>0$ and $M\geq 0$ such that for a.e. $t\in[a,b]$, $F(t,\gamma_*(t),\xi)\leq M$ for all $\xi\in T_{\gamma_*(t)}M$ with $\|\xi\|=\rho$.
\end{itemize}
Then the minimizer $\gamma_*$ is Lipschitz. Moreover, if $\|\dot \gamma_*(t)\|>\rho$, we take $v=\|\dot \gamma_*(t)\|/\rho>1$ in (\ref{W}), then
\begin{equation*}
  F\left(t,\gamma_*(t),\rho\frac{\dot{\gamma}_*(t)}{\|\dot \gamma_*(t)\|}\right)\geq \rho\frac{\Theta(\|\dot \gamma_*(t)\|)}{\|\dot \gamma_*(t)\|}-\|p\|_\infty.
\end{equation*}
Therefore $\|\dot \gamma_*(t)\|\leq \max\{\rho,R\}$ where  $R:=\inf\{s:\  \rho\frac{\Theta(s)}{s}>M+\|p\|_\infty\}$.
\end{lemma}
\section{Proof of Lemma \ref{Bol}}\label{apcc}

When $H(x,u,p)$ is superlinear in $p$, it is well-known that the functional $\mathbb L^t$ admits minimizers in $X_t(x)$. It remains to prove the existence of minimizers of $\mathbb L^t$ when $H(x,u,p)$ is coercive in $p$. Define
\begin{equation*}
  \mathbb L^t_n(\gamma)=\varphi(\gamma(0))+\int_0^t L_n(\gamma(s),v(\gamma(s),s),\dot{\gamma}(s))ds,
\end{equation*}
where $L_n$ is defined as in Section \ref{H3}. Then each $\mathbb L^t_n$ admits minimizers in $X_t(x)$. To prove the existence of the minimizers of $\mathbb L^t(\gamma)$, we define
\begin{equation*}
  m(r):=\inf_{x\in M}\left(\inf_{\|\dot x\|\geq r}L_1(x,0,\dot x)\right),\quad \forall r\geq 0.
\end{equation*}
It is clear that the function $m(r)$ is superlinear, and
\begin{equation*}
\begin{aligned}
  m(\|\dot x\|)&\leq L_n(x,0,\dot x)\leq L_n(x,u,\dot x)+\lambda |u|
  \\ &\leq L(x,u,\dot x)+\lambda |u|,\quad \forall n\in\mathbb N,\ \forall (x,u,\dot x)\in TM\times \mathbb R.
\end{aligned}
\end{equation*}
For any sequence $\gamma_n$ in $X_t(x)$ with $\lim_n \mathbb L^t(\gamma_n)<+\infty$, we have $\sup_n\int_0^tm(\|\dot \gamma_n\|)ds<+\infty$, so $\gamma_n$ admits a weakly sequentially converging subsequence. By Lemma \ref{TM}, the functionals $\mathbb L^t$ and $\mathbb L^t_n$ are sequentially weakly lower semicontinuous on $X_t(x)$. Since $X_t(x)$ is a metric space, the functionals $\mathbb L^t$ and $\mathbb L^t_n$ are also lower semicontinuous. Note that $\{\mathbb L^t_n\}_{n\in \mathbb{N}}$ is an increasing sequence, and converges pointwisely to $\mathbb L^t$ on $X_t(x)$. Both $\mathbb L^t$ and $\mathbb L^t_n(\gamma)$ are lower semicontinuous.  We conclude that $\Gamma-\lim_{n\rightarrow +\infty} \mathbb L_n^t=\mathbb L^t$ on $X_t(x)$ by Lemma \ref{inc}.

If the minimizers $\gamma_n$ of $\mathbb L_n^t$ are contained in a compact subset of $X_t(x)$, then by Lemma \ref{inf}, one can obtain that  $\mathbb L^t$ admits a minimum point on $X_t(x)$. It remains to show that there exists a compact set in $X_t(x)$ such that all minimizers $\gamma_n$ are contained in this set. Consider the set
\begin{equation*}
  K_t(x):=\left\{\gamma\in X_t(x):\ \int_0^tm(\|\dot \gamma\|)ds\leq \|\phi\|_\infty+\mathbb Kt+2\lambda Kt\right\},
\end{equation*}
where $\mathbb K:=\sup_{x\in M}L(x,0,0)$ and $K:=\|v(x,t)\|_\infty$. The set $K_t(x)$ is weakly sequentially compact in $W^{1,1}([0,t],M)$. According to \cite[Theorem 2.13]{One}, $K_t(x)$ is compact in $X_t(x)$. For the constant curve $\gamma_x\equiv x$, we have
\begin{equation*}
  \int_0^tm(\|\dot \gamma_x\|)ds\leq\mathbb L^t_n(\gamma_x)+\lambda Kt\leq \mathbb L^t(\gamma_x)+\lambda Kt\leq \|\phi\|_\infty+\mathbb Kt+2\lambda Kt.
\end{equation*}
Therefore $\gamma_x$ is contained in $K_t(x)$. Similarly, for minimizers $\gamma_n$, we have
\begin{equation*}
\begin{aligned}
  \int_0^tm(\|\dot \gamma_n\|)ds&\leq\mathbb L^t_n(\gamma_n)+\lambda Kt\leq \mathbb L^t_n(\gamma_x)+\lambda Kt
  \\ &\leq \mathbb L^t(\gamma_x)+\lambda Kt\leq \|\phi\|_\infty+\mathbb Kt+2\lambda Kt.
\end{aligned}
\end{equation*}
Thus, all $\gamma_n$ are contained in $K_t(x)$.
\qed

\section{Proof of Lemma \ref{3.1}}\label{apbb}

\begin{proof}We first prove Item (1). According to (LIP) and the Lipschitz continuity of $v(x,t)$ on $M\times [0,T]$, for each $\tau\in [0,t]$, the map $s\mapsto L(\gamma(\tau),v(\gamma(\tau),s),\dot{\gamma}(\tau))$ satisfies the condition ($\diamondsuit$), where $k\equiv\lambda\|\partial_t v(x,t)\|_\infty$. By Lemma \ref{Lip6.3}, for every $(x,t)\in M\times[0,T]$, the minimizers of $u(x,t)$ are Lipschitz. However, the Lipschitz constant depends on the end point $(x,t)$. We aim to show that for $(x',t')$ sufficiently close to $(x,t)$, the Lipschitz constant of the minimizers of $u(x',t')$ is independent of $(x',t')$.

For any $r>0$, if $d(x,x')\leq r$ and $|t-t'|\leq r/2$, where $t\geq r>0$, we denote by $\gamma(s;x,t)$ and $\gamma(s;x',t')$ the minimizers of $u(x,t)$ and $u(x',t')$ respectively. Then we have
\begin{equation*}
\begin{aligned}
  u(x',t')=&\varphi(\gamma(0;x',t'))+\int_0^{t'}L(\gamma(s;x',t'),v(\gamma(s;x',t'),s),\dot{\gamma}(s;x',t'))ds
  \\ \leq& \varphi(\gamma(0;x,t))+\int_0^{t-r}L(\gamma(s;x,t),v(\gamma(s;x,t),s),\dot{\gamma}(s;x,t))ds\\
  &+\int_{t-r}^{t'}L(\alpha(s),v(\alpha(s),s),\dot \alpha(s))ds,
\end{aligned}
\end{equation*}
where $\alpha:[t-r,t']\rightarrow M$ is a geodesic connecting $\gamma(t-r;x,t)$ and $x'$ with constant speed. Noticing that
\begin{equation*}
  \|\dot \alpha\|\leq \frac{1}{t'-(t-r)}\bigl{(}d(\gamma(t-r;x,t),x)+d(x,x')\bigl{)}\leq 2\left(\frac{1}{r}\int_{t-r}^t\|\dot\gamma(s;x,t)\|ds+1\right),
\end{equation*}
we obtain that
\begin{equation*}
  \int_0^{t'}L(\gamma(s;x',t'),v(\gamma(s;x',t'),s),\dot{\gamma}(s;x',t'))ds
\end{equation*}
has a bound depending only on $(x,t)$ and $r$. By (SL), there exists a constant $M(x,t,r)>0$ such that
\begin{equation*}
  \int_0^{t'}\|\dot{\gamma}(s;x',t')\|ds\leq M(x,t,r),
\end{equation*}
where $t'\geq t-r/2>0$. It means $\|\dot{\gamma}(s;x',t')\|$ are equi-integrable. Therefore, for $(x',t')$ sufficiently close to $(x,t)$, there exists a constant $R(x,t,r)>0$ and $s_0\in[0,t']$ such that $\|\dot \gamma(s_0;x',t')\|\leq R(x,t,r)$. By Lemma \ref{W4.1}, there exists an absolutely continuous function $p(t;x',t')$ satisfying $|p'(t;x',t')|\leq \lambda\|\partial_tv(x,t)\|_\infty$ such that
\begin{equation*}
\begin{aligned}
  L(&\gamma(s;x',t'),v(\gamma(s;x',t'),s),\frac{\dot \gamma(s;x',t')}{\theta})\theta
  \\ &-L(\gamma(s;x',t'),v(\gamma(s;x',t'),s),\dot \gamma(s;x',t'))\geq p(s;x',t')(\theta-1),\quad \forall \theta>0.
\end{aligned}
\end{equation*}
One can take $\theta=2$ and $t=s_0$ to obtain the upper bound of $p(s_0)$, and take $\theta=1/2$ and $t=s_0$ to obtain the lower bound of  $p(s_0)$. Note that $p'(t)$ is bounded. We finally obtain the bound of $\|p(t)\|_\infty$, which is independent of $(x',t')$. Since $L(x,u,\dot x)$ satisfies (SL), according to Lemma \ref{Lip6.3} and taking $\rho=1$, we have
\begin{equation*}
  L(\gamma(s;x',t'),v(\gamma(s;x',t'),s),\frac{\dot \gamma(s;x',t')}{\|\dot \gamma(s;x',t')\|})\geq \frac{\Theta(\|\dot \gamma(s;x',t')\|)}{\|\dot \gamma(s;x',t')\|}-\|p(s;x',t')\|_\infty.
\end{equation*}
Therefore, for $(x',t')$ sufficiently close to $(x,t)$, the minimizers $\gamma(s;x',t')$ have a Lipschitz constant independent of $(x',t')$.

\vspace{1ex}

In order to prove Item (2), we first show that $u(x,t)$ is locally Lipschitz in $x$. For any $\delta>0$, fix $(x_0,t)\in M\times[\delta,T]$ and $x$, $x'\in B(x_0,\delta/2)$. We denote by $d_0=d(x,x')\leq \delta$ the Riemannian distance between $x$ and $x'$. Then
\begin{equation*}
\begin{aligned}
  u(x',t)-u(x,t)\leq& \int_{t-d_0}^{t}L(\alpha(s),v(\alpha(s),s),\dot \alpha(s))ds
  \\ &-\int_{t-d_0}^{t}L(\gamma(s;x,t),v(\gamma(s;x,t),s),\dot{\gamma}(s;x,t))ds,
\end{aligned}
\end{equation*}
where $\gamma(s;x,t)$ is a minimizer of $u(x,t)$ and $\alpha:[t-d_0,t]\rightarrow M$ is a geodesic connecting $\gamma(t-d_0;x,t)$ and $x'$ with constant speed. By Lemma \ref{3.1} (1), if $x\in B(x_0,\delta/2)$, the bound of $\|\dot \gamma(s;x,t)\|$ depends only on $x_0$ and $\delta$. Noticing that
\begin{equation*}
  \|\dot \alpha(s)\|\leq \frac{d(\gamma(t-d_0;x,t),x')}{d_0}\leq \frac{d(\gamma(t-d_0;x,t),x)}{d_0}+1,
\end{equation*}
and  \[d(\gamma(t-d_0;x,t),x)\leq \int_{t-d_0}^t\|\dot \gamma(s;x,t)\|ds,\] the bound of $\|\dot \alpha(s)\|$  depends only on $x_0$ and $\delta$. Exchanging the role of $(x,t)$ and $(x',t)$, one obtain that \[|u(x,t)-u(x',t)|\leq J_1d(x,x'),\] where $J_1$ depends only on $x_0$ and $\delta$. Since $M$ is compact, we conclude that for $t\in(0,T]$, the value function $u(\cdot,t)$ is Lipschitz on $M$.

We are now going to show the locally Lipschitz continuity of $u(x,t)$ in $t$. Given $t_0\geq 3\delta/2$ and $t$, $t'\in[t_0-\delta/2,t_0+\delta/2]$, without any loss of generality, we assume $t'>t$. Then
\begin{equation*}
\begin{aligned}
  u(x,t')-u(x,t)\leq& u(\gamma(t;x,t'),t)-u(x,t)
  \\ &+\int_t^{t'}L(\gamma(s;x,t'),v(\gamma(s;x,t'),s),\dot \gamma(s;x,t'))ds,
\end{aligned}
\end{equation*}
where the bound of $\|\dot \gamma(s;x,t')\|$ depends only on $t_0$ and $\delta$. We have shown that for $t\geq \delta$, the following holds
\begin{equation*}
  u(\gamma(t;x,t'),t)-u(x,t)\leq J_1d(\gamma(t;x,t'),x)\leq J_1\int_t^{t'}\|\dot \gamma(s;x,t')\|ds\leq J_2(t'-t).
\end{equation*}
Thus, $u(x,t')-u(x,t)\leq J_3(t'-t)$, where $J_3$ depends only on $t_0$ and $\delta$. The condition $t'<t$ is similar. We conclude the locally Lipschitz continuity of $u(x,\cdot)$ on $(0,T]$.

\vspace{1ex}

At last, we prove Item (3).  We first prove that $u(x,t)$ is continuous at $t=0$. For each $\varphi\in C(M)$, there is a sequence $\varphi_m\in \text{Lip}(M)$ converging to $\varphi$  uniformly. We take  $\varphi$ and $\varphi_m$ as the initial functions in (\ref{u}), and denote by $u(x,t)$ and $u_m(x,t)$ the corresponding value functions respectively. Since $v(x,t)$ is fixed, by the non-expansiveness of the Lax-Oleinik semigroup, we have $\|u(x,t)-u_m(x,t)\|_\infty\leq \|\varphi-\varphi_m\|_\infty$. Thus, without loss of generality, we assume the initial function to be Lipschitz in the following discussion. Take a constant curve $\alpha(t)\equiv x$. Let $\gamma:[0,t]\to M$ be a minimizer of $u(x,t)$. It is obvious that
\begin{equation*}
  u(x,t)=\varphi(\gamma(0))+\int_0^tL(\gamma(s),v(\gamma(s),s),\dot \gamma(s))ds\leq \varphi(x)+\int_0^tL(x,v(x,s),0)ds,
\end{equation*}
so $\limsup_{t\rightarrow 0^+}u(x,t)\leq \varphi(x)$. By (SL), there exists a constant $C>0$ such that
\begin{align*}
  \int_0^tL(\gamma(\tau),v(\gamma(\tau),\tau),\dot \gamma(\tau))d\tau&\geq \int_0^t\|\partial_x \varphi\|_\infty\|\dot \gamma(\tau)\|d\tau+Ct\\
  &\geq \|\partial_x \varphi\|_\infty d(\gamma(0),\gamma(t))+Ct,
\end{align*}
which implies that
\begin{equation*}
  \int_0^tL(\gamma(\tau),v(\gamma(\tau),\tau),\dot \gamma(\tau))d\tau+\varphi(\gamma(0))\geq \varphi(x)+Ct.
\end{equation*}
Therefore $\liminf_{t\rightarrow 0^+}u(x,t)\geq \varphi(x)$. Combining with Lemma \ref{3.1} (2),  $u(x,t)$ is continuous on $M\times [0,T]$.

By a standard argument, one can show that the value function $u(x,t)$ is a  solution of (\ref{cu}). We omit the details.

\end{proof}

\section{Weak KAM solutions and viscosity solutions}\label{swv}

%


Following Fathi \cite{Fa08}, one can extend the definitions of backward and forward weak KAM solutions of  equation \eqref{hj} by using absolutely continuous calibrated curves instead of $C^1$ curves.
\begin{definition}\label{bws}
A function $u_-\in C(M)$ is called a backward weak KAM solution of \eqref{hj} if the following hold.
\begin{itemize}
\item [(1)] For each absolutely continuous curve $\gamma:[t',t]\rightarrow M$, we have
\begin{equation*}
  u_-(\gamma(t))-u_-(\gamma(t'))\leq \int_{t'}^{t}L(\gamma(s),u_-(\gamma(s)),\dot \gamma(s))ds.
\end{equation*}
The above condition reads that $u_-$ is dominated by $L$ and denoted by $u_-\prec L$.

\item [(2)] For each $x\in M$, there exists an absolutely continuous curve $\gamma_-:(-\infty,0]\rightarrow M$ with $\gamma_-(0)=x$ such that
\begin{equation*}
  u_-(x)-u_-(\gamma_-(t))=\int_t^0L(\gamma_-(s),u_-(\gamma_-(s)),\dot \gamma_-(s))ds,\quad \forall t<0.
\end{equation*}
The curves satisfying the above equality are called $(u_-,L,0)$-calibrated curves.
\end{itemize}
\end{definition}
A forward weak KAM solution of \eqref{hj} can be defined in a similar manner. We omit the details.
%

\begin{proposition}\label{u<LLip}
If $u\prec L$, then $u$ is a Lipschitz function  on $M$.
\end{proposition}
\begin{proof}
For each $x,y\in M$, let $\alpha:[0,d(x,y)/\delta]\rightarrow M$ be a geodesic of length $d(x,y)$, with constant speed $\|\dot \alpha\|=\delta$ and connecting $x$ and $y$. Then
\begin{equation*}
  L(\alpha(s),u(\alpha(s)),\dot{\alpha}(s))\leq \bar{C}+\lambda \|u\|_\infty,\quad \forall s\in [0,d(x,y)/\delta].
\end{equation*}
Then by $u\prec L$ we have
\begin{equation*}
  u(y)-u(x)\leq \int_0^{d(x,y)/\delta}L(\alpha(s),u(\alpha(s)),\dot{\alpha}(s))ds\leq \frac{1}{\delta}(\bar{C}+\lambda \|u\|_\infty) d(x,y).
\end{equation*}
Exchanging the role of $x$ and $y$, we get the Lipschitz continuity of $u$.
\end{proof}

By \cite[Corollary 8.3.4]{Fa08}, we have
\begin{proposition}\label{aesub}
Suppose $H(t,x,p)$ is a continuous function,  and it is coercive and convex in $p$, then $u(x,t)$ is a subsolution of $\partial_tu+H(t,x,\partial_xu)=0$ if $u(x,t)$ is locally Lipschitz and $\partial_tu+H(t,x,\partial_xu)\leq 0$ holds almost everywhere.
\end{proposition}

\begin{proposition}\label{fix}
The following statements are equivalent:
\begin{itemize}
\item [(1)] $u_-$ is a viscosity solution of (\ref{E});

\item [(2)] $u_-$ is a fixed point of $T^-_t$;

\item [(3)] $u_-$ is a backward weak KAM solution.
\end{itemize}
Similarly,  the following statements are also equivalent:
\begin{itemize}
\item [(i)] $-u_+$ is a viscosity solution of $H(x,-u,-\partial_xu)=0$;

\item [(ii)] $u_+$ is a fixed point of $T^+_t$;

\item [(iii)] $u_+$ is a forward weak KAM solution.
\end{itemize}
\end{proposition}
\begin{proof}
By Theorem  \ref{m1}, (2) implies (1). We show that (1) implies (2). Since $u_-$ is a viscosity solution of (\ref{E}), the function $u(x,t):=u_-(x)$ is the viscosity solution of (\ref{C}) with the initial condition $u(x,0)=u_-(x)$. By the comparison theorem, we have $u(x,t)=T^-_t u_-(x)$, which implies $u_-=T^-_t u_-$.

Now we show that (3) implies (2). According to the definition of the backward weak KAM solution, for $u_-\in\mathcal S_-$ we have
\begin{equation*}
  u_-(x)=\inf_{\gamma(t)=x}\left\{u_-(\gamma(0))+\int_0^t L(\gamma(\tau),u_-(\gamma(\tau)),\dot{\gamma}(\tau))d\tau\right\},
\end{equation*}
where the infimum is taken in the class of absolutely continuous curves. We show $u_-(x)\leq T^-_tu_-(x)$. The opposite direction is similar.  Assume \[u_-(x)>T^-_tu_-(x).\] Let $\gamma:[0,t]\rightarrow M$ with $\gamma(t)=x$ be a minimizer of $T^-_tu_-(x)$. Define
\begin{equation*}
  F(\tau):=u_-(\gamma(\tau))-T^-_\tau u_-(\gamma(\tau)).
\end{equation*}
Since $F(t)>0$ and $F(0)=0$, there is $s_0\in [0,t)$ such that $F(s_0)=0$ and $F(s)>0$ for $s\in(s_0,t]$. By definition we have
\begin{equation*}
  T^-_su_-(\gamma(s))=T^-_{s_0}u_-(\gamma(s_0))+\int_{s_0}^sL(\gamma(\tau),T^-_\tau u_-(\gamma(\tau)),\dot\gamma(\tau))d\tau,
\end{equation*}
and
\begin{equation*}
  u_-(\gamma(s))\leq u_-(\gamma(s_0))+\int_{s_0}^sL(\gamma(\tau),u_-(\gamma(\tau)),\dot\gamma(\tau))d\tau,
\end{equation*}
which implies
\begin{equation*}
  F(s)\leq \lambda\int_{s_0}^s F(\tau)d\tau.
\end{equation*}
By the Gronwall inequality, we conclude $F(s)\equiv 0$ for all $s\in[s_0,t]$, which contradicts $F(t)>0$.

It remains to show (2) implies (3). For each absolutely continuous curve $\gamma:[t',t]\rightarrow M$, we have
\begin{equation*}
\begin{aligned}
  &u_-(\gamma(t))-u_-(\gamma(t'))=T^-_tu_-(\gamma(t))-T^-_{t'}u_-(\gamma(t'))
  \\ &\leq \int_{t'}^{t}L(\gamma(s),T^-_su_-(\gamma(s)),\dot \gamma(s))ds=\int_{t'}^{t}L(\gamma(s),u_-(\gamma(s)),\dot \gamma(s))ds,
\end{aligned}
\end{equation*}
which implies $u_-\prec L$. We now show the existence of the $(u_-,L,0)$-calibrated curve. We define a sequence of absolutely continuous curves as follows: Let $\gamma_0(0)=x$ and $\gamma_n:[0,1]\rightarrow M$  be a minimizer of $T^-_1 u_-(\gamma_{n-1}(0))$ with $\gamma_n(1)=\gamma_{n-1}(0)$. We define $\gamma_-:(-\infty,0]\rightarrow M$ by $\gamma_-(-t):=\gamma_{[t]+1}([t]+1-t)$ for all $t>0$, which is also absolutely continuous. Here, $[t]$ stands for the greatest integer not greater than $t$. Then we have
\begin{equation*}
\begin{aligned}
  u_-(\gamma_-(-[t]))-u_-(\gamma_-(-t))
  &=T^-_1u_-(\gamma_{[t]+1}(1))-T^-_{[t]+1-t}u_-(\gamma_{[t]+1}([t]+1-t))\\
  &=\int_{[t]+1-t}^{1}L(\gamma_{[t]+1}(s),T^-_su_-(\gamma_{[t]+1}(s)),\dot \gamma_{[t]+1}(s))ds\\
  &=\int_{-t}^{-[t]}L(\gamma_-(s),u_-(\gamma_-(s)),\dot \gamma_-(s))ds.
\end{aligned}
\end{equation*}
Similarly, one can prove that for all $n=0,1,\dots$,
\[u_-(\gamma_-(-n))-u_-(\gamma_-(-n-1))=\int_{-n-1}^{-n}L(\gamma_-(s),u_-(\gamma_-(s)),\dot \gamma_-(s))ds.\]
We conclude that $\gamma_-:(-\infty,0]\rightarrow M$ is a $(u_-,L,0)$-calibrated curve.

The proof is now complete.
\end{proof}

\end{document}